\documentclass[11pt]{article}
\usepackage{amsmath}
\usepackage{amssymb}
\usepackage{amsthm}
\usepackage{authblk}
\usepackage[english]{babel}
\usepackage{a4wide}
\usepackage{overpic}
\usepackage{graphicx}
\usepackage{tikz}
\usepackage{subfigure}
\usepackage{url}
\usepackage{hyperref}
\hypersetup{
    colorlinks,%
    citecolor=black,
    filecolor=black,
    linkcolor=black,
    urlcolor=black
}

\usepackage{color}
\newtheorem{theorem}{Theorem}[section]
\newtheorem{Def}[theorem]{Definition}

\newtheorem{Thm}[theorem]{Theorem}

\newtheorem{Lem}[theorem]{Lemma}

\numberwithin{equation}{section}

\DeclareMathSizes{10}{10}{10}{10}

    \usepackage{color}

 \DeclareMathOperator*{\supp}{supp}
    
    \DeclareMathOperator{\repart}{Re}
    \DeclareMathOperator{\impart}{Im}
    \renewcommand{\Re}{\repart}
    \renewcommand{\Im}{\impart}
\DeclareMathOperator{\ord}{ord}

\author[1]{Daan Huybrechs\thanks{daan.huybrechs@cs.kuleuven.be}}
\author[2]{Arno B.J. Kuijlaars\thanks{arno.kuijlaars@wis.kuleuven.be}} 
\author[1]{Nele Lejon\thanks{nele.lejon@cs.kuleuven.be}} 
\affil[1]{{\small KU Leuven, Department of Computer Science,  Celestijnenlaan 200A, 3001 Leuven, Belgium}}
\affil[2]{{\small KU Leuven, Department of Mathematics, Celestijnenlaan 200B, 3001 Leuven, Belgium}}
\title{Zero distribution of complex orthogonal polynomials with respect to exponential weights}

\begin{document}
\maketitle
\begin{abstract}
We study the limiting zero distribution of orthogonal polynomials with respect to some 
particular exponential weights $e^{-nV(z)}$ along contours in the complex plane. We are especially interested in 
the question under which circumstances the zeros of the orthogonal polynomials accumulate on a single 
analytic arc (one cut case), and in which cases they do not. 
In a family of cubic polynomial potentials $V(z) = - \tfrac{iz^3}{3} + iKz$,  
we determine the precise values of $K$ for which we have the one cut case. 
We also prove the one cut case for a monomial quintic $V(z) = - \tfrac{iz^5}{5} $ on a contour
that  is symmetric in the imaginary axis.  
\end{abstract}

\section{Introduction\label{sect_intro}}
The subject of this paper is the study of the zeros of orthogonal polynomials with respect to 
a varying exponential weight of the form  $ e^{-nV(z)}$ for a  polynomial $V$. 
The weight is considered along an unbounded contour $\Gamma$ in the complex plane that connects
two sectors in the complex plane in which $\Re V(z) \to +\infty$.  
The polynomial $P_n$ is assumed to be monic of degree $n$, and it satisfies the orthogonality condition
\begin{align} \label{orthogonality}
	\int_\Gamma z^k P_{n}(z) e^{-nV(z)} dz=0,	\qquad \text{for } k=0,\ldots, n-1.
\end{align}
Note that \eqref{orthogonality} is an example of non-hermitian orthogonality with respect to a varying weight,
and so  the family  of polynomials $\{P_n(z)\}_n$ is not an orthogonal family in itself, 
since the orthogonality weight varies with $n$.
We are interested in the limiting distribution of the zeros of the polynomials $P_n$
as $n \to \infty$. 

Let $V$ be a polynomial of degree $d \geq 2$, which we will also call the potential.
Then there are $d$ sectors in the complex plane with angle $\frac{\pi}{d}$, 
such that $\Re V\to +\infty$ in each of these sectors. Let us number these 
sectors counterclockwise as  $S_1,S_2,\ldots, S_d$, as shown in Figure \ref{Fig_Sectors}
for the case $V(z) = -\frac{iz^3}{3}$.

The complementary sectors $S_1', S_2', \ldots, S_d'$ are sectors where $\Re V \to -\infty$. 
We label them such that $S_j'$ follows $S_j$ if we go around in the counterclockwise
direction. The complementary sectors are shaded in grey in Figure \ref{Fig_Sectors}.

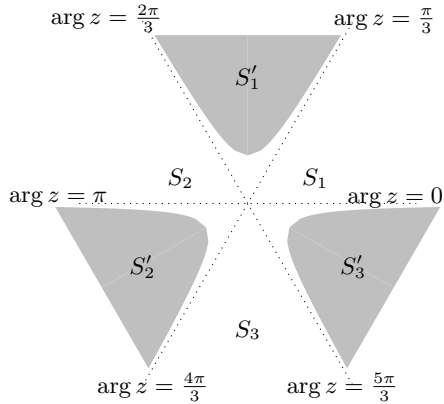
\begin{figure}[!t]
\centering
\begin{tikzpicture} [scale=0.8]
\tikzstyle{every node}=[font=\footnotesize]
\pgftransformscale{0.35 }{
\draw[dotted] (-8,0) -- (8,0);
\draw [dotted](-5,-8.660254040) --  (5,8.660254040);
\draw [dotted](-5,8.660254040) --  (5,-8.660254040);
\draw (3.2,1.2) node{$S_1$};
\draw (-3.2,1.2) node{$S_2$};

\fill[color=gray!50,domain=2.279507057:8,rotate=330] plot (\x,{sqrt(3)/3*sqrt((\x^2)-3*sqrt(3))})-- (8,0) ;
\fill[color=gray!50,domain=2.279507057:8,rotate=330] plot (\x,{-sqrt(3)/3*sqrt((\x^2)-3*sqrt(3))})-- (8,0) ;
\fill[color=gray!50,domain=2.279507057:8,rotate=90] plot (\x,{sqrt(3)/3*sqrt((\x^2)-3*sqrt(3))})-- (8,0) ;
\fill[color=gray!50,domain=2.279507057:8,rotate=90] plot (\x,{-sqrt(3)/3*sqrt((\x^2)-3*sqrt(3))})-- (8,0) ;
\draw (0,6) node{$S'_1$};
\fill[color=gray!50,domain=2.279507057:8,rotate=210] plot (\x,{sqrt(3)/3*sqrt((\x^2)-3*sqrt(3))})-- (8,0) ;
\fill[color=gray!50,domain=2.279507057:8,rotate=210] plot (\x,{-sqrt(3)/3*sqrt((\x^2)-3*sqrt(3))})-- (8,0) ;
\draw (-5,-3) node{$S'_2$};
\draw (0,-6) node{$S_3$};
\draw (5,-3) node{$S'_3$};
\draw(7,0.3)node {$\arg z=0$};
\draw(6.5,8.7)node{$\arg z=\frac{\pi}{3}$};
\draw(-6.7,8.8)node{$\arg z=\frac{2\pi}{3}$};
\draw(-9,0.2)node{$\arg z=\pi$};
\draw(-4.5,-8.7)node{$\arg z=\frac{4\pi}{3}$};
\draw(4.5,-8.7)node{$\arg z=\frac{5\pi}{3}$};
}

\end{tikzpicture}
\caption{The sectors $S_1,S_2,\ldots, S_d$ (white) and the complementary sectors 
$S'_1,S'_2,\ldots, S'_d$ (grey) for a potential of degree $d=3$.\label{Fig_Sectors}}
\end{figure}

For $1 \leq j, k \leq d$ with $j \neq k$, we use $\mathcal T_{j,k}$ to denote the set of 
contours in the complex plane that start in sector $S_j$ and end at infinity in sector $S_k$. By Cauchy's theorem
we have that  integrals such as 
\begin{align*}
\int_\Gamma z^k e^{-nV(z)}dz, \qquad \Gamma \in \mathcal T_{j,k}
\end{align*} do
not depend on the particular choice of $\Gamma$. In particular 
the orthogonal polynomial $P_n$ satisfying \eqref{orthogonality} does not depend on $\Gamma$
but only on the class $\mathcal T_{j,k}$.

It follows from recent work of Bertola \cite{Bert}, see also \cite{BerMo}, that the 
zeros of $P_n$ tend to  a union of analytic arcs, determined by a hyperelliptic curve 
that satisfies the so-called Boutroux condition. A different approach based on 
a max-min variational problem in logarithmic potential theory is due to Rakhmanov \cite{Rak},
and it was worked out in detail for the present setting in  \cite{KuSil}.
We give more details about this approach below. 

It is of  interest to determine the nature
of these arcs, and in particular to find out if the zeros go to one analytic arc (one cut case)
or to a union of two or more analytic arcs. 
We focus in this paper on two potentials, namely the third degree polynomial 
\begin{align} 
V(z) = -\frac{iz^3}{3}+iKz, \label{V1}
\end{align} 
with a real constant $K$, and the monic fifth degree polynomial
\begin{align}
	V(z)=-\frac{iz^5}{5}.\label{V2}
\end{align}
The interest in these particular potentials is inspired by the research of Dea\~no, 
Huybrechs and Kuijlaars \cite{DHK}, who showed that for the case \eqref{V1} with $K=0$ 
the zeros of the  polynomials $P_n$ accumulate on one contour.  This work is motivated
by a computational approach to oscillatory integrals based on steepest descent analysis
\cite{DH}.
The cubic model is also studied in detail in \cite{AMM1} with a combination
of analytical and numerical techniques.  See also \cite{AMM2, BleDea, KMM} for related
results.  The case of a quartic polynomial potential is studied
in detail in \cite{BerTov}.

The sectors $S_1, \ldots, S_d$ in which $\Re V \to +\infty$ are shown in
Figure \ref{Fig_admitted_sect_deg3} for the case of the cubic potential \eqref{V1}
and in Figure \ref{Fig_sect_Gamma_A} for the monic quintic potential \eqref{V2}.
We will restrict ourselves to cases that are symmetric with respect to the imaginary axis,
which is also the reason for the imaginary unit $i$ in \eqref{V1} and \eqref{V2}.
The extra symmetry yields that there is only one possibility in the cubic case, namely we have
to connect sectors $S_2$ and $S_1$ as shown in Figure \ref{Fig_admitted_sect_deg3}, that
is, we consider the family $\mathcal T_{2,1}$, which we will simply call $\mathcal T$ when
dealing with the cubic potential.

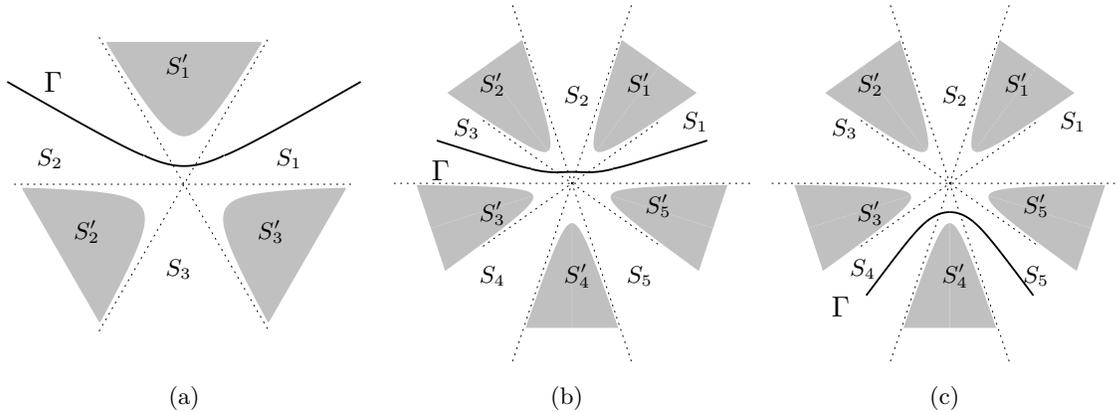
\begin{figure}

\subfigure[]{\begin{overpic}[scale=0.25,unit=1mm]
 {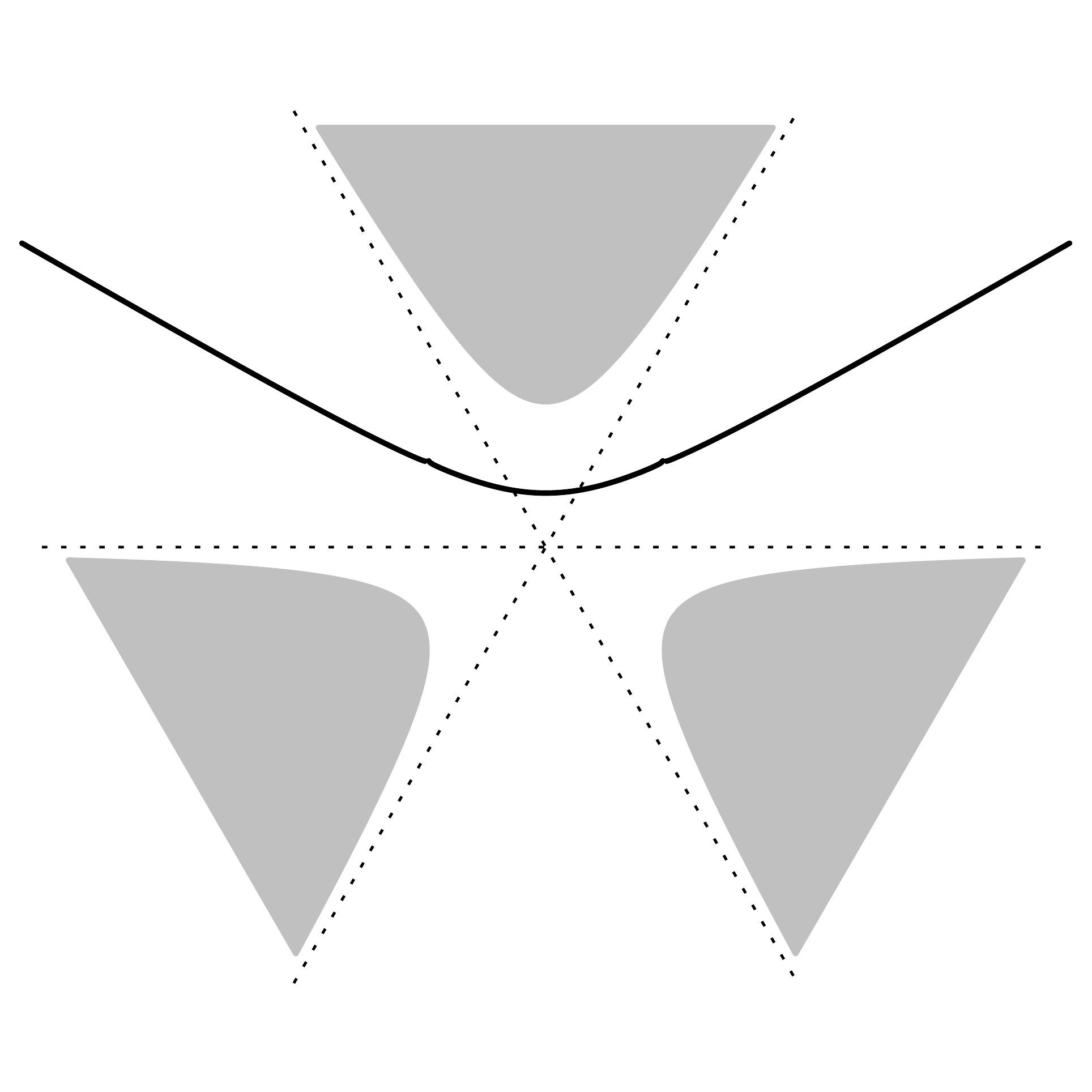}
   \put(12,75){$\Gamma$}
\put(75,55){\footnotesize $S_1$}
\put(45,80){\footnotesize $S'_1$}
   \put(10,55){\footnotesize $S_2$}
   \put(20,35){\footnotesize $S'_2$}
   \put(45,25){\footnotesize $S_3$}
   \put(70,35){\footnotesize $S'_3$}
\end{overpic}
\label{Fig_admitted_sect_deg3}}
 \subfigure[]{\begin{overpic}[scale=.25,unit=1mm]%
      {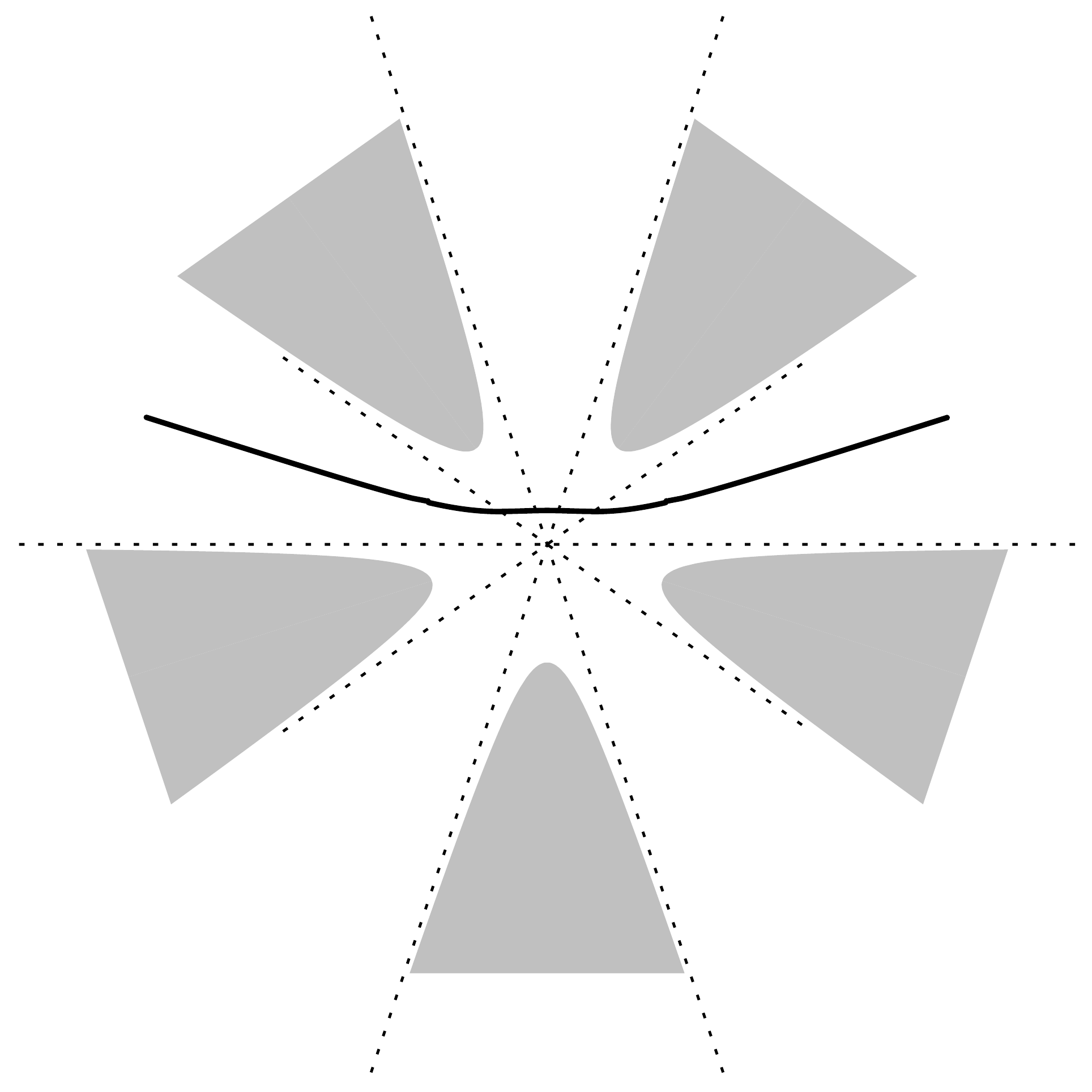}
     \put(80,65){\footnotesize $S_1$}
     \put(18,63){\footnotesize $S_3$}
     \put(65,75){\footnotesize $S'_1$}
      \put(25,75){\footnotesize $S'_2$}
     \put(48,72){\footnotesize $S_2$}
      \put(12,51){$\Gamma$}
     \put(25,40){\footnotesize $S'_3$}
      \put(25,23){\footnotesize $S_4$}
      \put(48,23){\footnotesize $S'_4$}
      \put(65,23){\footnotesize $S_5$}
      \put(70,42){\footnotesize $S'_5$}
 \end{overpic}\label{Fig_sect_Gamma_A}}
      \subfigure[]{\begin{overpic}[scale=.25,unit=1mm]%
      {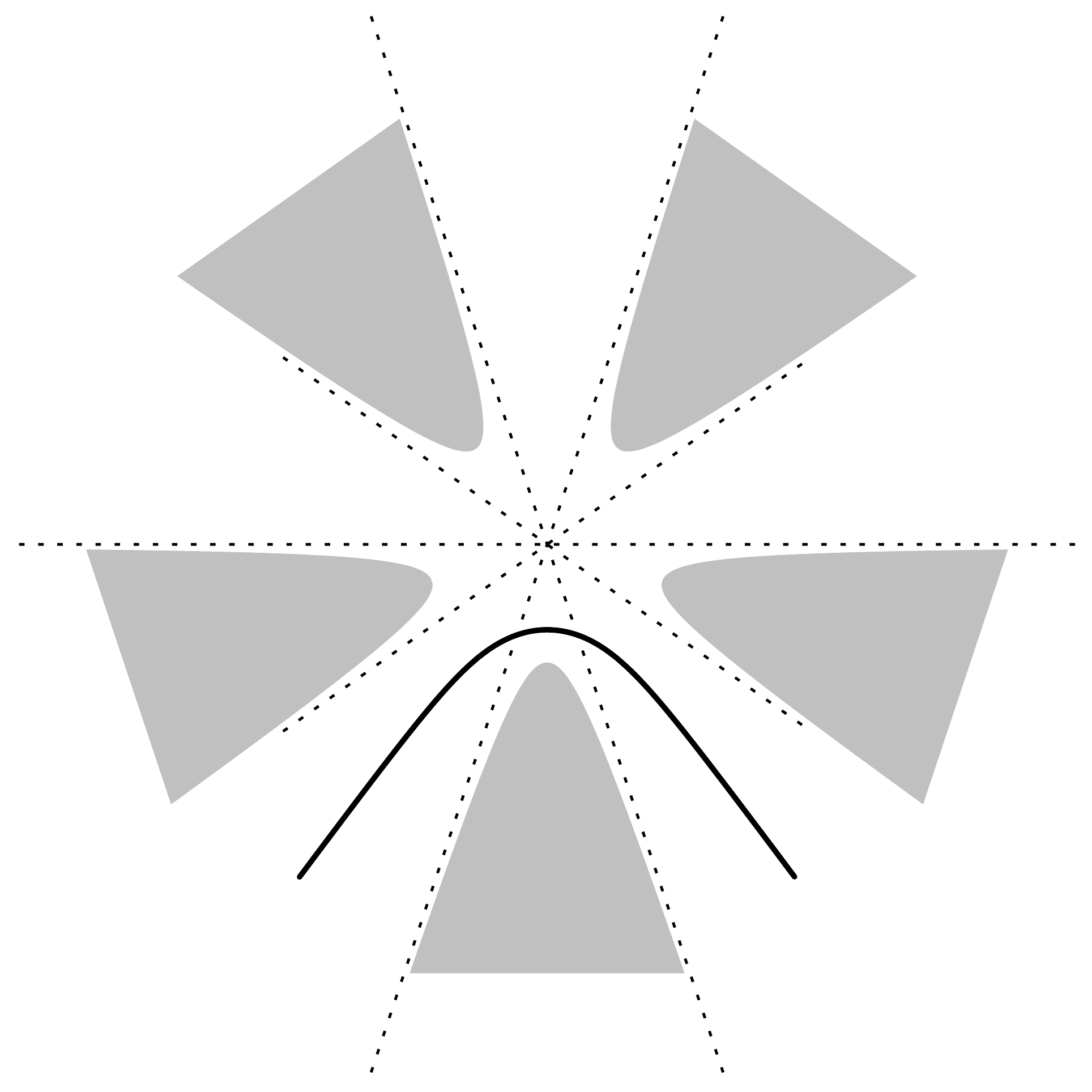}
     \put(80,65){\footnotesize $S_1$}
     \put(18,63){\footnotesize $S_3$}
     \put(65,75){\footnotesize $S'_1$}
      \put(25,75){\footnotesize $S'_2$}
     \put(48,72){\footnotesize $S_2$}
      \put(18,14){$\Gamma$}
     \put(25,40){\footnotesize $S'_3$}
      \put(23,25){\footnotesize $S_4$}
      \put(48,23){\footnotesize $S'_4$}
      \put(70,23){\footnotesize $S_5$}
      \put(70,42){\footnotesize $S'_5$}
 \end{overpic}\label{Fig_sect_Gamma_B}}
 \caption{Examples of curves along which we study orthogonality, for the potential 
of degree $d=3$ (a) and for degree $d=5$ (b and c).\label{Fig_admitted_sect} }
\end{figure}

Our new result for this case is as follows.

\begin{Thm}[Cubic case] \label{Thm_single_cut_V1}
Let $V(z)=-\frac{iz^3}{3}+iKz$ with $K \in \mathbb R$, and let $\mathcal T = \mathcal T_{2,1}$
be the family of contours that connect the two sectors $S_2$ and $S_1$ as shown in Figure \ref{Fig_admitted_sect_deg3}.

Then there is a unique critical value $K^*$ such that 
\begin{enumerate}
\item[\rm (a)] If $K<K^*$, the zeros of the orthogonal polynomials accumulate on one analytic arc.
\item[\rm (b)] If $K=K^*$, the zeros accumulate on one arc, which is not analytic at the point of intersection with the imaginary axis.
\item[\rm (c)] If $K>K^*$, the zeros accumulate on two disjoint arcs.
\end{enumerate}

The constant $K^*$ is determined by solving the equation
\begin{equation} \label{vequation}  
	-3 v \ln(2v) + 6v\ln(\sqrt{4+2v}+2) + (2- 2v) \sqrt{4+2v} = 0, \qquad v > 0 
	\end{equation}
which has a unique solution $v^* \approx 3.150037074 > 0$ and then 
putting\footnote{This and other numerical calculations are performed by \texttt{Maple}.}
\begin{equation} \label{defKstar} 
	K^* =  (v^*)^{1/3} - (v^*)^{-2/3} \approx 1.0005424. 
	\end{equation}
\end{Thm}

In Figures \ref{gamma_sub}, \ref{gamma_crit}, and \ref{gamma_super} we plotted 
the analytic arcs on which the zeros accumulate for the three cases $K < K^*$, $K= K^*$ and $K > K^*$.
The endpoints of the arcs are denoted by $z_1$, $z_2$ (in case $K \leq K^*$) and by
$z_1, \ldots, z_4$ (in case $K > K^*)$. The arcs have analytic continuations that
are also shown in the figures. The arcs together with their analytic continuation
can be used for the contour $\Gamma \in \mathcal T$, and this contour is in a
certain sense ideal for
the orthogonality \eqref{orthogonality}. 

\begin{figure}[!t]
\centering
\subfigure[$K<K^*$ ]{
\begin{overpic}[scale=.24,unit=1mm]%
{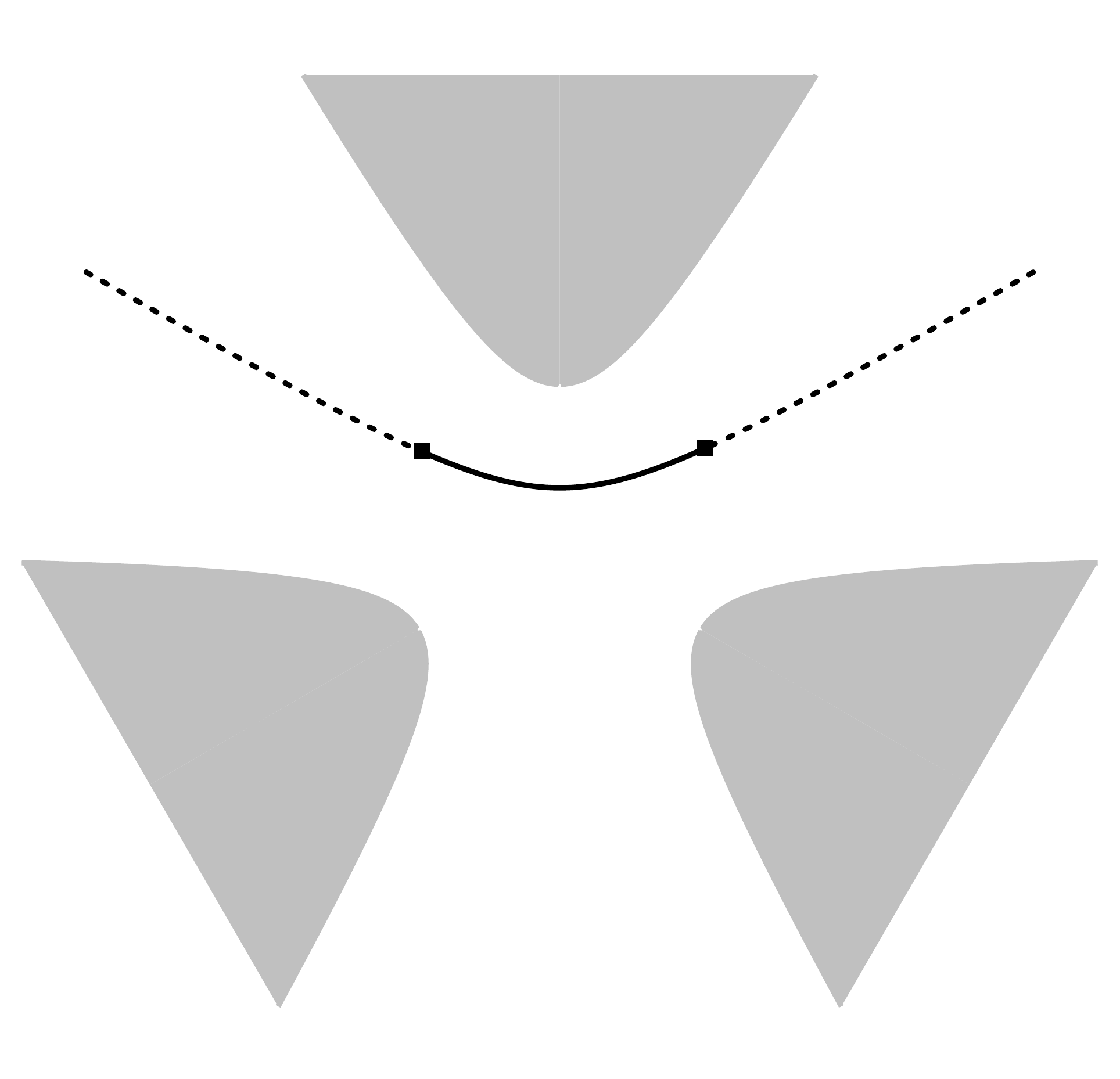}
\put(33,52){\footnotesize $z_1$}
   \put(64,52){\footnotesize $z_2$}
   \put(44,80){\footnotesize $S'_1$}
   \put(75,72){\footnotesize $S_1$}
   \put(20,72){\footnotesize $S_2$}
   \put(20,25){\footnotesize $S'_2$}
   \put(44,15){\footnotesize $S_3$}
   \put(75,25){\footnotesize $S'_3$}
   \put(12,60){\footnotesize $\Gamma$}\label{gamma_sub}
\end{overpic}
} 
\subfigure[$K=K^*$]{\begin{overpic}[scale=.24,unit=1mm]%
     {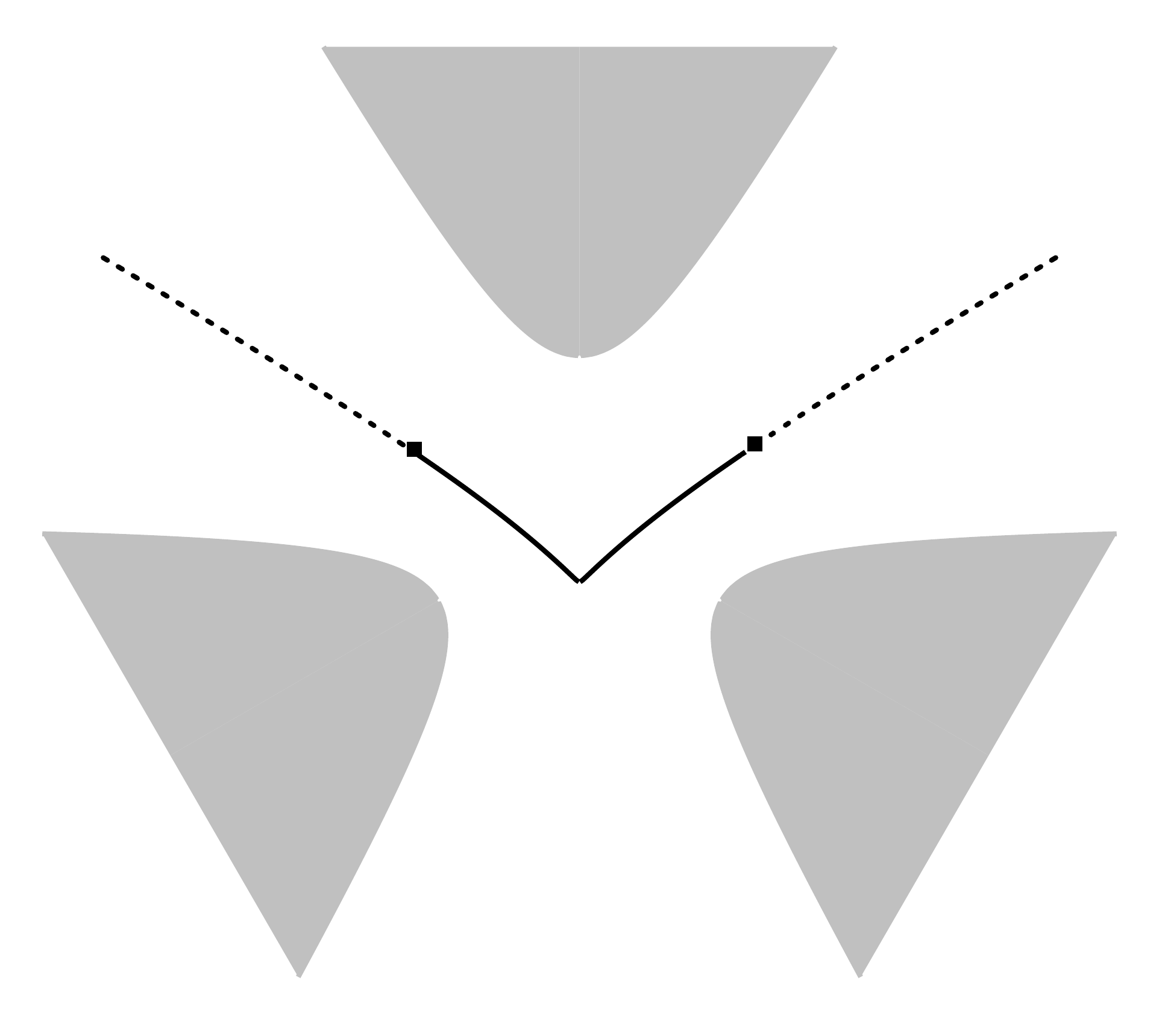}
     \put(26,50){\footnotesize $z_1$}
   \put(72,50){\footnotesize $z_2$}
      \put(43,77){\footnotesize $S'_1$}
   \put(75,72){\footnotesize $S_1$}
   \put(20,72){\footnotesize $S_2$}
   \put(20,25){\footnotesize $S'_2$}
   \put(44,15){\footnotesize $S_3$}
   \put(75,25){\footnotesize $S'_3$}
   \put(12,56){\footnotesize $\Gamma$}
\end{overpic}\label{gamma_crit}
} 
\subfigure[$K>K^*$ ]{
\begin{overpic}[scale=.24,unit=1mm]%
      {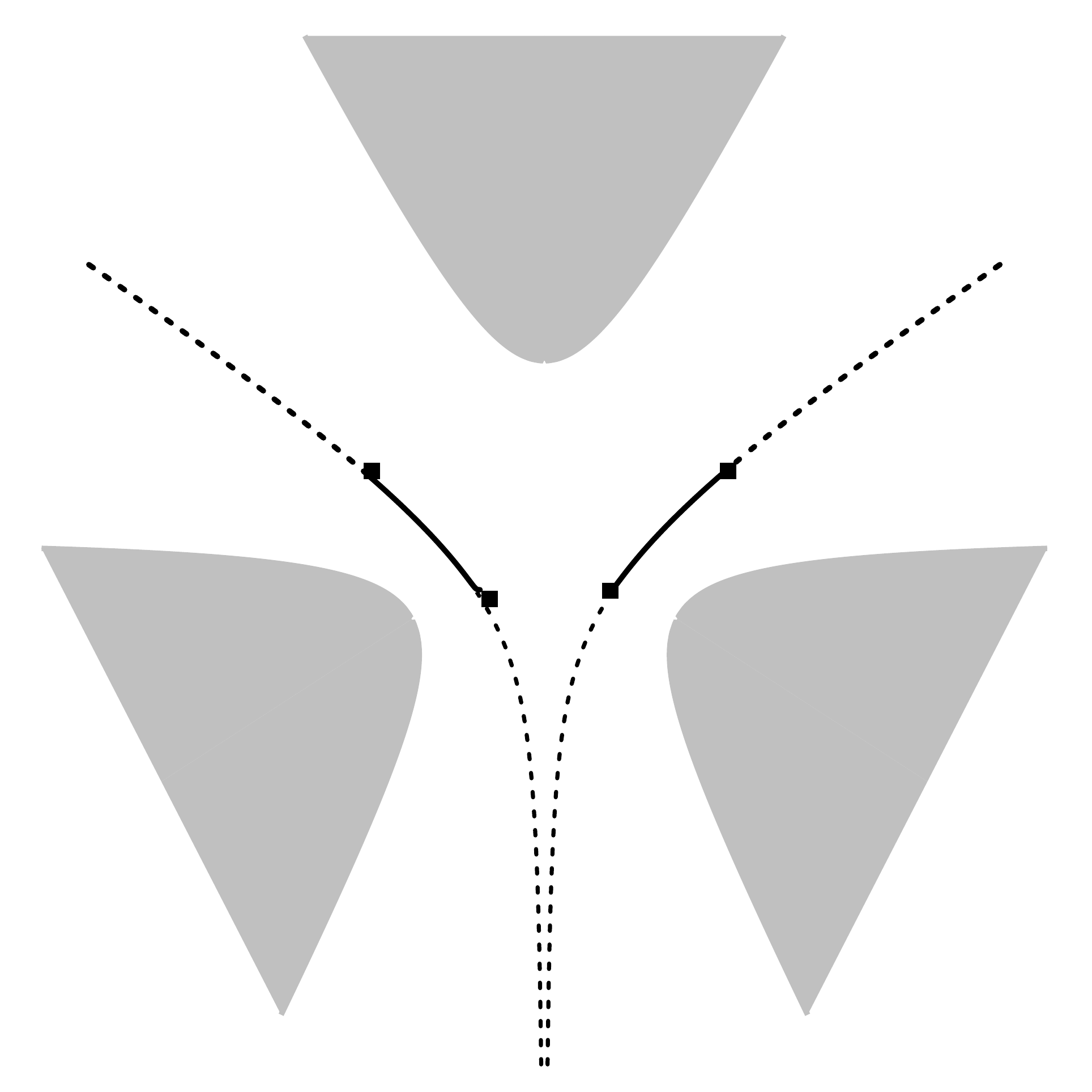}
  \put(26,56){\footnotesize $z_1$}
    \put(72,56){\footnotesize $z_2$}
    \put(38,42){\footnotesize $z_3$}
    \put(56,40){\footnotesize $z_4$}
          \put(43,81){\footnotesize $S'_1$}
   \put(75,72){\footnotesize $S_1$}
   \put(20,72){\footnotesize $S_2$}
   \put(20,25){\footnotesize $S'_2$}
   \put(41,15){\footnotesize $S_3$}
   \put(75,25){\footnotesize $S'_3$}
   \put(12,56){\footnotesize $\Gamma$}
\end{overpic}
\label{gamma_super}}
\caption{The figures illustrates Theorem \ref{Thm_single_cut_V1}. The solid lines are 
the analytic arcs on which 
the zeros of $P_n$ accumulate as $n \to \infty$. 
Panel (a) shows the result for $K=0$ as a typical example for $K<K^*$, 
panel (b) shows the arc for $K = K^*$ and panel (c)  shows the result for $K=2$, as an example of $K > K^*$. 
The dotted lines are the analytic extensions that together with the analytic arcs form the contour $\Gamma$. 
\label{Scurves_V1}}
\end{figure}
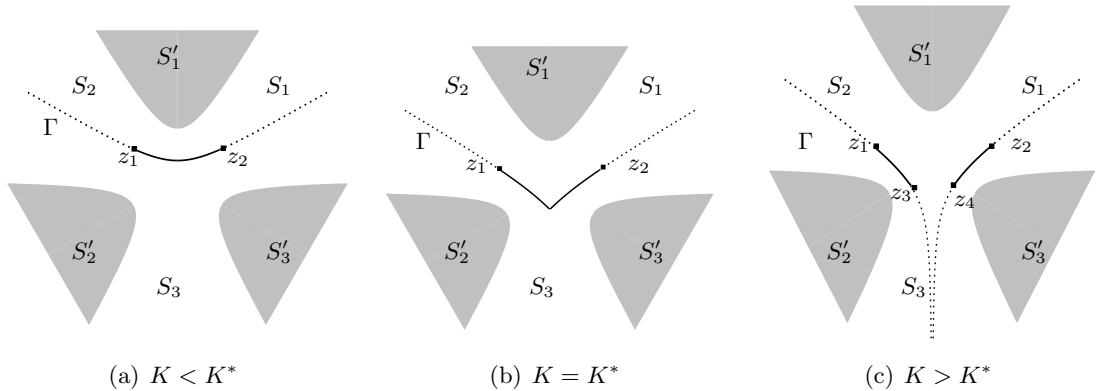

In the quintic case there are two possible combinations of sectors that respect the symmetry
in the imaginary axis. We can use either $\mathcal T_{3,1}$ or $\mathcal T_{4,5}$ as shown
in Figures \ref{Fig_sect_Gamma_A} and \ref{Fig_sect_Gamma_B}. 
In both cases we find that zeros of the orthogonal polynomials accumulate on one arc.

 \begin{theorem}[Quintic case] \label{Thm_single_cut_V2}
  Let $V(z) = - \frac{iz^5}{5}$. Then for both choices $\mathcal T_{3,1}$ and $\mathcal T_{4,5}$
	the  zeros of the orthogonal polynomials accumulate on one analytic arc.
  \end{theorem}

The analytic arcs on which the zeros of the orthogonal polynomials 
accumulate are plotted in Figure \ref{Scurves_V2} for the cases $\mathcal T_{3,1}$
and $\mathcal T_{4,5}$. The analytic extensions of the arcs are also shown in the figure.
 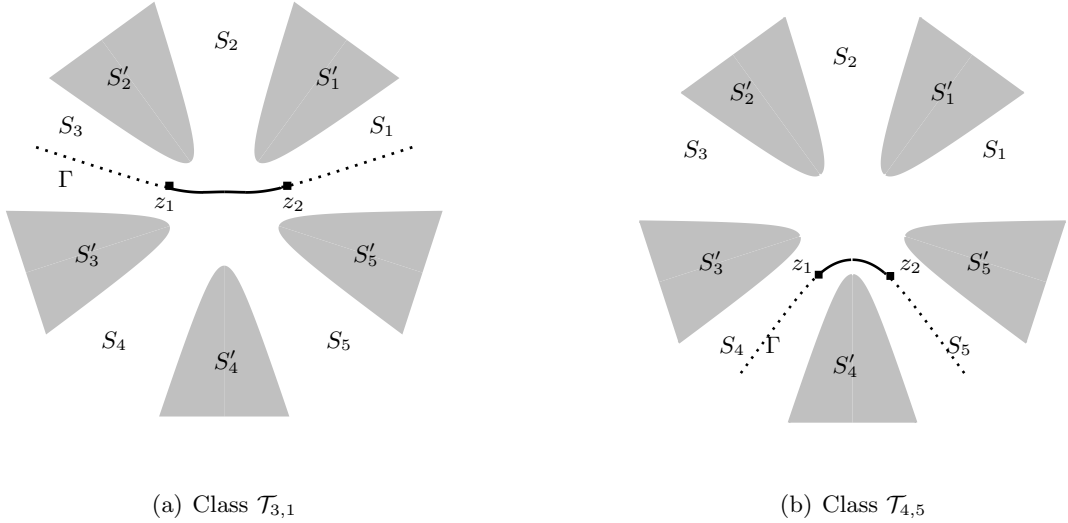
\begin{figure}[!h]
 \begin{center}
 \subfigure[Class $\mathcal{T}_{3,1}$]{
\begin{overpic}[scale=.35,unit=1mm]%
      {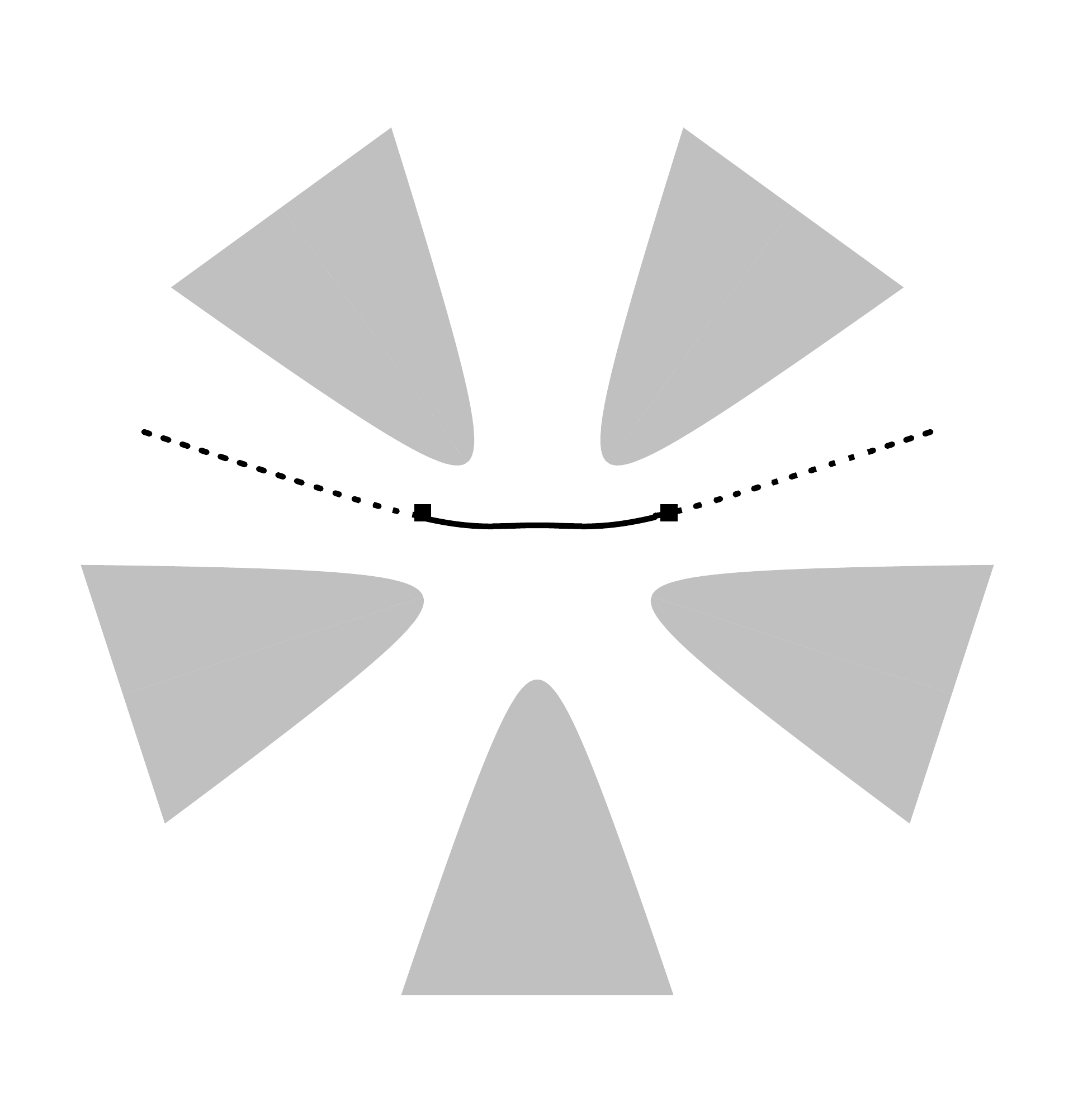}
  \put(33,50){ \footnotesize $z_1$}
    \put(57,50){ \footnotesize $z_2$}
    \put(46,80){\footnotesize $S_2$}
        \put(26,73){\footnotesize $S'_2$}
         \put(65,73){\footnotesize $S'_1$}
         \put(75,64){\footnotesize $S_1$}
         \put(17,64){\footnotesize $S_3$}
         \put(17,54){\footnotesize $\Gamma$}
         \put(20,40){\footnotesize $S'_3$}
				 \put(25,24){\footnotesize $S_4$}
         \put(46,20){\footnotesize $S'_4$}
				 \put(67,24){\footnotesize $S_5$}
         \put(72,40){\footnotesize $S'_5$}
\end{overpic}
}
\hspace{1cm}
\subfigure[Class $\mathcal{T}_{4,5}$]{
\begin{overpic}[scale=.35,unit=1mm]%
      {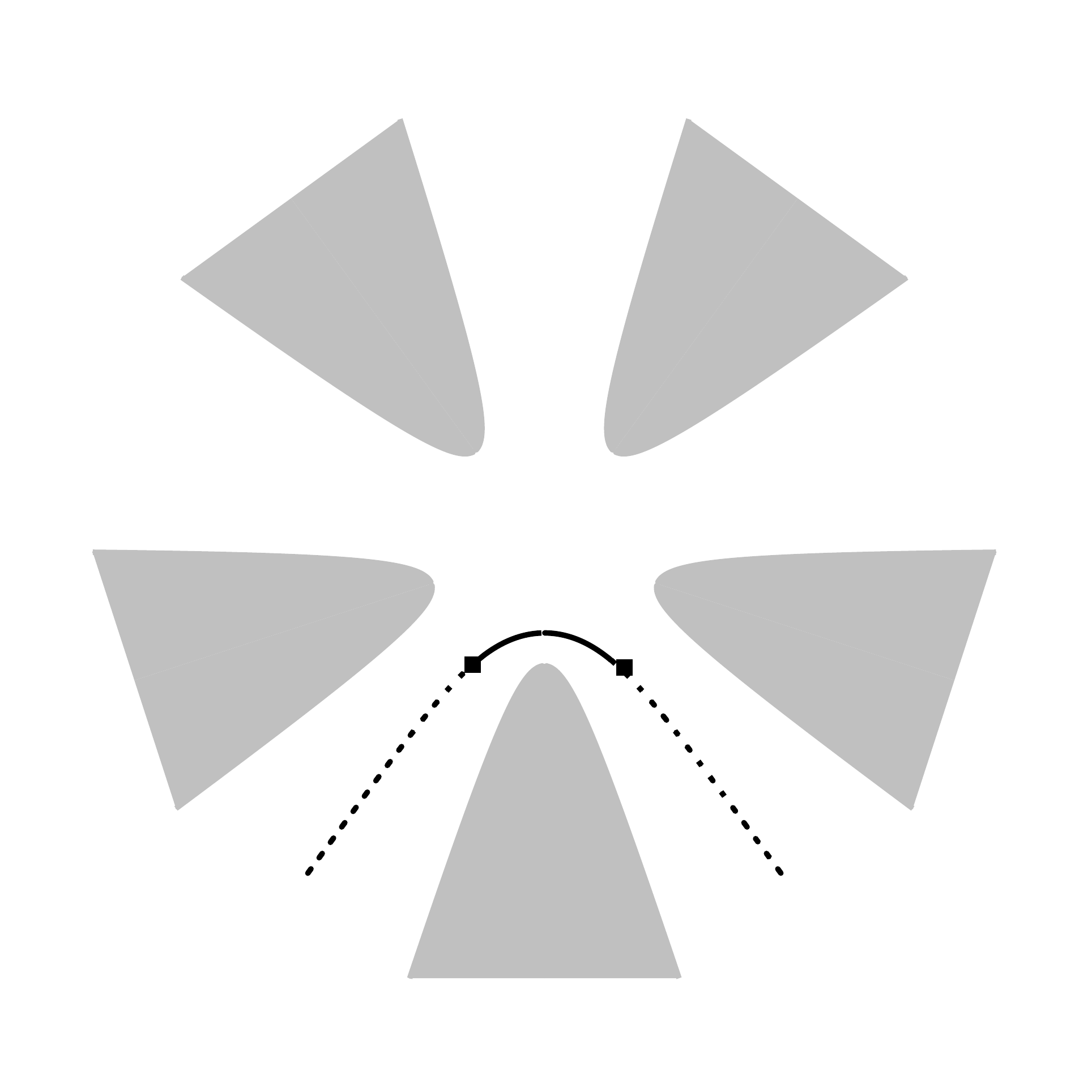}
        \put(38,40){\footnotesize $z_1$}
    \put(59,40){\footnotesize $z_2$}
        \put(46,80){\footnotesize $S_2$}
        \put(26,73){\footnotesize $S'_2$}
         \put(65,73){\footnotesize $S'_1$}
         \put(75,62){\footnotesize $S_1$}
         \put(17,62){\footnotesize $S_3$}
         \put(33,24){\footnotesize $\Gamma$}
         \put(20,40){\footnotesize $S'_3$}
				 \put(24,24){\footnotesize $S_4$}
         \put(46,20){\footnotesize $S'_4$}
				 \put(68,24){\footnotesize $S_5$}
         \put(72,40){\footnotesize $S'_5$}
\end{overpic}
}
\end{center}
  \caption{The figure shows the analytic arc on which the zeros of the orthogonal polynomials accumulate
	for the potential $V(z)=-\frac{iz^5}{5}$ and for the two choices of $\mathcal T$. 
	The extended contour belongs to the class $\mathcal T_{3,1}$ in the left panel (a) 
	and to $\mathcal T_{4,5}$ in the right panel (b).\label{Scurves_V2}}
 \end{figure}

It is an open problem to determine the nature of the analytic arcs for higher degree potentials.
It will, for example, be interesting to find out if for any monomial $V(z) = - \frac{iz^d}{d}$
and any class of contours $\mathcal T$ that respects the  symmetry  in the imaginary axis
the zeros accumulate on one single arc, or not.

The proofs of Theorems \ref{Thm_single_cut_V1} and \ref{Thm_single_cut_V2} rely on a result
of Gonchar and Rakhmanov \cite{GR}, see Theorem \ref{theoremGR} below, which says that
if a contour $\Gamma \in \mathcal T$ has the $S$-property in the external field $\Re V$,
then the zeros of the polynomials $P_n$ tend to  $\Gamma$, and then the equilibrium measure of $\Gamma$
in the external field $\Re V$ is the limit of the normalized zero counting measures as $n \to \infty$.
There is always a contour with the $S$-property in $\mathcal T$.  The support of its
equilibrium measure in external field $\Re V$ consists of critical trajectories of 
a quadratic differential $-Q(z) dz^2$ where $Q$ is a certain polynomial that is determined
by $V$ and $\mathcal T$. We give necessary background
on $S$-curves and quadratic differentials in Section \ref{sect_prelim}.

The proofs of Theorems \ref{Thm_single_cut_V1} and \ref{Thm_single_cut_V2}
are in Sections \ref{sect_deg3} and  \ref{sect_deg5}. In both proofs, we start
by collecting properties of $Q$ that are satisfied assuming we are in the one cut case.
This gives us only one candidate for $Q$ in the cubic case of Theorem \ref{Thm_single_cut_V1}
and two candidates in the quintic case of Theorem \ref{Thm_single_cut_V2}.
We show that a critical 
trajectory of $-Q(z)dz^2$ connects the two simple zeros of $Q$ (in the
cubic case we have to restrict to $K < K^*$) and that this critical trajectory
has an analytic continuation to a contour in the desired class $\mathcal T$.
The proof of the $S$-property is then finished by an argument that we do not give here,
since it is completely analogous to the proof of \cite[Theorem 2.2]{DHK}.

\section{Preliminaries}\label{sect_prelim}

\subsection{$S$-curve in external field}

The zeros of the orthogonal polynomial accumulate on a contour (or union of contours) 
that is an $S$-curve in the external field $\Re V$.
To explain what this means we need certain notions from logarithmic potential theory. 
The following concepts are well known, see \cite{ST}.

 \begin{Def} 
 The logarithmic energy in external field $\Re V$ of a measure $\nu$ is defined as
 $$ E_{V}(\nu) =  \iint \log \frac{1}{|s-t|} d\nu(s) d\nu(t) + \int \Re V(s)  d\nu(s). $$
The equilibrium energy  of a contour $\Gamma$ in the external field $\Re V$ is:
\begin{align} \label{EVGamma}
	\mathcal{E}_V(\Gamma) = \inf_{\nu\in\mathcal{M}(\Gamma)} E_V(\nu),
\end{align}
where $\mathcal{M}(\Gamma)$ denotes the space of Borel probability measures on $\Gamma$. 
\end{Def}

If $\Re V(z)/\log(1+ |z|^2)$ tends to $+\infty$ as $|z| \to \infty$ on $\Gamma$,
then there is a unique minimizing measure for \eqref{EVGamma}, which is called the equilibrium
measure of $\Gamma$ in external field $\Re V$. We will denote it by $\mu_{\Gamma}$. The support
of $\mu_{\Gamma}$ is a compact subset of $\Gamma$. The equilibrium measure 
$\mu = \mu_{\Gamma}$ is characterized by the variational conditions which say that
for some $\ell = \ell_{\Gamma}$,
\begin{equation} 
	 2 U^{\mu}(z) + \Re V(z)  \begin{cases} = \ell, \qquad z \in \supp(\mu), \\
	 	\geq \ell, \qquad z \in \Gamma,
	 	\end{cases} \end{equation}
where 
\[ U^{\mu}(z) = \int \log \frac{1}{|z-s|} d\mu(s), \qquad z \in \mathbb C, \]
denotes the logarithmic potential of $\mu$.

\begin{Def}
The contour $\Gamma$ has the $S$-property in the external field $\Re V$ if
the equilibrium measure $\mu = \mu_{\Gamma}$ is supported on a finite number
of analytic arcs, and if on the interior of each analytic arc we have
\[ \frac{\partial}{\partial n_+} \left( 2 U^{\mu} + \Re V \right)
	= \frac{\partial}{\partial n_-} \left( 2 U^{\mu} + \Re V \right), \]
where $\frac{\partial}{\partial n_{\pm}}$ denote the derivatives in the normal directions on $\Gamma$.
A contour with the $S$-property is also called an $S$-curve.
\end{Def}

The following result is due to Gonchar and Rakhmanov \cite{GR}.
\begin{theorem} \label{theoremGR}
Suppose that a contour $\Gamma$ has the $S$-property in the external field $\Re V$.
Let $z_{1n}, z_{2n}, \ldots, z_{nn}$ denote
the zeros of the orthogonal polynomial $P_n$ that is characterized by \eqref{orthogonality}. Then
\[ \lim_{n \to \infty} \frac{1}{n} \sum_{j=1}^n \delta_{z_{jn}} = \mu_{\Gamma} \]
in the sense of weak convergence of probability measures. Here $\mu_{\Gamma}$ is the equilibrium
measure of $\Gamma$ in the external field $\Re V$. 
\end{theorem}
In view of this theorem it is a natural question to ask whether the class $\mathcal T_{j,k}$ contains 
a contour with the $S$-property in the external field $\Re V$.  
The following result was obtained in \cite{KuSil} by adapting the results
of \cite{MFR} and \cite{Rak} 
to the situation of a polynomial external field. See also \cite{MRS}.

\begin{theorem} \label{theoremRMF}
For every polynomial $V$ of degree $d \geq 2$ and every choice of sectors $S_j$
and $S_k$ with $1 \leq j \neq k \leq d$, the following holds.
\begin{enumerate}
\item[\rm (a)] There is a contour $\Gamma \in \mathcal T_{j,k}$ such that
\begin{equation} \label{EVoptimal} 
	\mathcal E_V(\Gamma) = \sup_{\Gamma' \in \mathcal T_{jk}} \mathcal E_V(\Gamma'). 
	\end{equation}
\item[\rm (b)] The contour $\Gamma$ has the $S$-property
in the external field $\Re V$ (and so the zeros of $P_n$ tend to the support
of $\mu_{\Gamma}$ by Theorem \ref{theoremGR}).
\item[\rm (c)] The function
\begin{equation} \label{defQ}
	Q(z) = \left(\frac{1}{2} V'(z) - \int \frac{d\mu_{\Gamma}(s)}{z-s} \right)^2, 
	\qquad z \in \mathbb C \setminus \supp\mu_{\Gamma} 
	\end{equation}
is  a polynomial of degree $\deg Q = 2d-2$.
\item[\rm (d)] The support of $\mu_{\Gamma}$ consists of critical trajectories of the quadratic differential
$- Q(z) dz^2$, and
\[ d\mu_{\Gamma}(s) = \frac{1}{\pi i} Q^{1/2}(s) ds \]
holds on each trajectory.
\end{enumerate}
\end{theorem}

Recall that a curve $\gamma$ containing the point $z_0$ is a trajectory (or horizontal trajectory) for $-Q(z) dz^2$ if 
\begin{equation} \label{horizontal} 
	\int_{z_0}^z Q^{1/2}(s) ds \in i \mathbb R, \qquad \text{for }  z \in \gamma, 
	\end{equation}
and $\gamma$ is a critical trajectory if it passes through a zero of $Q$, see \cite{Str}.
For us a trajectory is always a maximal trajectory. Since $Q$ is a polynomial, we have
that a trajectory is either an analytic arc connecting two zeros of $Q$,
or an unbounded analytic curve connecting a zero of $Q$ with infinity, or a two-sided
unbounded curve. A  vertical trajectory of $-Q(z) dz^2$ going through $z_0$
is  a curve $\gamma$ such that   
\begin{equation} \label{vertical} 
	\int_{z_0}^z Q^{1/2}(s) ds \in \mathbb R, \qquad \text{for }  z \in \gamma. 
	\end{equation}
That is, vertical trajectories are the usual (horizontal) trajectories of
$Q(z) dz^2$.

The combination of \eqref{EVGamma} and \eqref{EVoptimal} characterizes the $S$-curve in
terms of a max-min problem. For a given contour $\Gamma$
we minimize the energy in the external field, and then we maximize over the family $\mathcal T_{j,k}$ 
to obtain the curve with the $S$-property.

For a given $V$ and family $\mathcal T_{j,k}$ our task is to identify
the polynomial $Q$ from \eqref{defQ}. Since $\mu_{\Gamma}$ is a probability measure,
we find from \eqref{defQ} that
\begin{align} \nonumber
	Q(z) & = \left( \frac{1}{2} V'(z) - \frac{1}{z} + \mathcal{O}(z^{-2})\right)^2 \\
	& = \frac{1}{4} V'(z)^2-\frac{V'(z)}{z}+\mathcal{O}\left(z^{d-3}\right)
	\qquad \text{ as } z \to \infty,
	\label{Qasympt}
\end{align}
which is not enough to specify $Q$ except in the case $d=2$. For $d \geq 3$
there are $d-2$ unspecified constants in the polynomial \eqref{Qasympt}. 

\subsection{Trajectories of a quadratic differential}

In this subsection we assume that $Q$ is a polynomial of degree $2d-2$ and
we collect some properties of the trajectories of the quadratic
differential $-Q(z) dz^2$. 
The local structure of critical trajectories of $-Q(z) dz^2$ is well-understood.

\begin{Lem} \label{m_traject} 
Suppose $z_0$ is a zero of the polynomial $Q$ of order $m$. Then there are $m+2$ 
critical trajectories emanating from $z_0$
at equal angles $\psi_0, \ldots, \psi_{m+1}$. The angles  are solutions of the equation   
\begin{equation}\label{dir_crit_alg}
	(m+2) \psi = \pi-\arg Q^{(m)}(z_0)   \mod  2 \pi, \qquad
	j=0, \ldots, m+1.
\end{equation}
\end{Lem}
\begin{proof}
This result is proved in \cite{Str}. 
\end{proof}

So in particular, if $z_0$ is a simple zero then three critical trajectories emanate from $z_0$
at equal angles given by 
\begin{equation}\label{dir_crit1}
	3\psi = \pi-\arg Q'(z_0)\quad \mod 2\pi.
\end{equation}
If $z_0$ is a double zero then four critical trajectories are emerging from $z_0$ at angles satisfying: 
\begin{equation}\label{dir_crit2}
	4\psi= \pi-\arg Q''(z_0) \quad  \mod 2\pi.
\end{equation}

In the situation of Theorem \ref{theoremRMF} we note the following about the zeros of $Q$.
\begin{Lem} If $Q$ is as in Theorem \ref{theoremRMF} then zeros of odd multiplicity of 
$Q$ are contained in $\supp \mu_{\Gamma}$.
\end{Lem}
\begin{proof}
From \eqref{defQ} we know that $Q$ has an analytic and single-valued square root in $\mathbb C \setminus \supp \mu_{\Gamma}$.
Then the lemma follows, since we clearly cannot have a single-valued square root in a neighborhood of a zero of $Q$ of
odd multiplicity.
\end{proof}

Two trajectories cannot intersect, except at a zero of $Q$. Also a trajectory cannot
be closed, since we are dealing with a polynomial $Q$, as given in the following lemma.
 
\begin{Lem} \label{No_loop}
Let $Q$ be a polynomial of degree $\geq 2$. Then there cannot be a closed contour that is
a trajectory, or a finite union of trajectories, of $-Q(z) dz^2$.
\end{Lem}

\begin{proof}
See e.g. Lemmas 8.3 and 8.4 in \cite{Pom}. 
\end{proof}

We need to control the behavior of trajectories at infinity. 
Since $Q$ is a polynomial of degree $2d-2$, the point at infinity
is a pole of the quadratric differential $-Q(z) dz^2$ of order $2d + 2 \geq 4$.

\begin{Lem} \label{lem:angles}
Let $Q(z) = - c^2 z^{2d-2} + \mathcal{O}(z^{2d-3})$ be a polynomial of degree $2d-2$ with $d \geq 2$. 
Then any unbounded trajectory of $-Q(z) dz^2$ that ends at infinity does so at one of $2d$ possible angles
$\theta_j$ where
\begin{equation} \label{angles} 
\theta_j = - \frac{1}{d} \arg c + \frac{j \pi}{d}, \qquad j = 0, \ldots, 2d-1.
\end{equation}
\end{Lem}
\begin{proof}
See  \cite[Theorem 7.4]{Str} for the statement that there are $2d$ possible directions at infinity,
forming equal angles. 

The trajectories are curves along which $\Im \int^z i Q^{1/2}(s) ds$ is constant. 
If $Q(z) = - c^2 z^{2d-2} + \mathcal{O}(z^{2d-3})$
then this means that $\Im( \frac{c}{2d-1} z^{2d} + \mathcal{O}(z^{2d-1}))$ is constant along any unbounded
trajectory, which gives rise to the possible angles \eqref{angles}.
\end{proof}

Useful information about the global behavior of trajectories is
contained in Teichm\"uller's lemma \cite[Theorem 14.1]{Str}.
This lemma involves the notion of a $Q$-polygon, which in this context
is a simple closed curve on the Riemann sphere that is composed of a finite
number of horizontal and vertical trajectories of the quadratic differential $-Q(z)dz^2$.

The order $\ord(z)$ of a point $z$ on the Riemann sphere is defined by 
\[ \ord(z) = \begin{cases} 
	n & \text{ if $z$ is a zero of $Q$ of order $n$}, \\
	-n & \text{ if $z$ is a pole of $Q$ of order $n$}, \\
	0 & \text{ if $z$ is not a zero or a pole}.
	\end{cases} \]
Then Teichm\"uller's lemma (for the special case of a polynomial
quadratic differential) says the following.

\begin{theorem} \label{Teichmuller}
Let $Q$ be a polynomial, and let $\Omega$ be a domain on the Riemann sphere that is bounded by a $Q$-polygon.
Then
\begin{equation} \label{Teichmuller_sum}
	\sum_j \left( 1 - \varphi_j \frac{n_j+2}{2\pi} \right) = 2 + \sum_i n_i,
	\end{equation}
where the sum on the left is over all vertices $z_j$ where $n_j$ is the order of $z_j$
and $\varphi_j \in [0,2\pi]$ is the interior angle of $\Omega$ at $z_j$, and
the sum on the right is over all interior zeros and poles $z_i$ in $\Omega$, 
and $n_i$ is the order of $z_i$. 
\end{theorem}

\begin{proof}
See \cite[Theorem 14.1]{Str}.
\end{proof}
Teichm\"uller's lemma is also used in the recent paper \cite{AMMT} 
to determine the structure of critical trajectories for 
the limiting behavior of zeros of Laguerre polynomials with varying complex parameters. 

One easy consequence of \eqref{Teichmuller_sum} is the following.
If a domain $\Omega$ is bounded by a trajectory $\gamma$ that
is unbounded in both directions, 
and if $\Omega$ does not contain any zeros of $Q$ then $\gamma$ extends to
infinity in two consecutive directions at infinity.
The angles are given by $\theta_j$ and $\theta_j + \frac{\pi}{d}$ for some $j$, 
if $\deg Q = 2d-2$.
Indeed, in this situation the right-hand side of \eqref{Teichmuller_sum} is $2$,
and only the point at infinity contributes to the sum in the left-hand side.
Since infinity is a pole of order $2d+2$, we have $n_{\infty} = -2d-2$. Then, if 
$\varphi_{\infty}$ is the angle at infinity,  \eqref{Teichmuller_sum} gives us $\varphi_{\infty} = \frac{\pi}{d}$
as claimed.

The same conclusion holds if $\Omega$ has no zeros of $Q$ and if it is 
bounded by two unbounded trajectories emanating from a zero of order $m$ 
and making an interior angle of $\frac{2\pi}{m+2}$ at the zero. 

If we put 
\begin{equation} \label{eq:D}	
D(z) = \frac{1}{\pi i} \int_{z^*}^z Q(s)^{1/2} ds 
\end{equation}
where $Q(s)^{1/2}$ is an analytic branch of the square root defined in a
neighborhood of infinity, then horizontal trajectories $\gamma$ are characterized by 
\[ \Im D(z) = \text{const } \qquad \text{ on } \gamma, \]
while $\Re D$ is constant on vertical trajectories.

\begin{Lem} \label{differentD}
Let $D$ be as in \eqref{eq:D} where $Q^{1/2}$ is an analytic square root
defined in a neighborhood $U$ of infinity. 
Suppose  $\gamma_1$ and $\gamma_2$ are two unbounded trajectories of $-Q(z) dz^2$ 
that end at the same angle  at infinity. Suppose $ \Im D(z) = c_1$   on $\gamma_1 \cap U$  and
$\Im D(z) = c_2$ on $\gamma_2 \cap U$. If $\gamma_1 \neq \gamma_2$ then $c_1 \neq c_2$.
\end{Lem}
\begin{proof}
Suppose that the two trajectories $\gamma_1$ and $\gamma_2$  
end at infinity at asymptotic angle $\theta_{j_0}$ for some $j_0$, see \eqref{angles}.

The same local structure at infinity holds true for the vertical trajectories,
i.e., for the trajectories of $Q(z) dz^2$. The unbounded vertical trajectories 
end at asymptotic angles $\theta_j + \frac{\pi}{2d}$, $j=-d+1, \ldots, d$.
Also $\Re D(z)$ is constant on a vertical trajectory.

Take a vertical  trajectory that is unbounded on both sides, and tends
to infinity at angles $\theta_{j_0} - \frac{\pi}{2d}$ and $\theta_{j_0} + \frac{\pi}{2d}$.
We may assume that the vertical trajectory is close enough to infinity so that it
does not contain or enclose any zeros of $Q$. Hence it does not contain any zeros of 
$D'$, since $D'(z) = \frac{1}{\pi i} Q(z)^{1/2}$.
Since $\Re D(z)$ is constant on the vertical  trajectory, this implies
that $\Im D(z)$ is strictly monotonic along the vertical trajectory.

Both $\gamma_1$ and $\gamma_2$  intersect the vertical trajectory,
and if $\gamma_1 \neq \gamma_2$ this will be at different points where $\Im D$
has distinct values and therefore $c_1 \neq c_2$.
\end{proof}

In the situation where $V(z) = - \frac{i}{d} z^d + \mathcal{O}(z^{d-1})$ and $Q$ satisfies \eqref{defQ},
then $Q(z) = - \frac{1}{4} z^{2d-2} + \mathcal{O}(z^{2d-3})$, so that the asymptotic angles \eqref{angles}
are given by
\begin{equation} \label{angles2} 
	\theta_j =  \frac{j\pi}{d}, \qquad j = -d+1, \ldots, d.
	\end{equation}

\section{Proof of Theorem \ref{Thm_single_cut_V1}\label{sect_deg3}}

\subsection{Introduction}
We are considering the cubic case
\[ V(z) = - \frac{i z^3}{3} + i K z \]
with a parameter $K \in \mathbb R$.
In view of \eqref{Qasympt} we are looking for a polynomial  $Q$ of degree $4$ satisfying
\begin{equation} \label{Q3expansion}
\begin{aligned} 
	Q(z) & =  \frac{1}{4} (- i z^2 + i K)^2  - \frac{-i z^2 + iK}{z} + \mathcal{O}(1) \\
	 & =  -\frac{1}{4}z^4 + \frac{K}{2}z^2+iz+C
	 \end{aligned}
	 \end{equation} 
for some constant $C$.
Since we are interested in the family $\mathcal T = \mathcal T_{2,1}$ there is a symmetry
with respect to the imaginary axis. It means that the constant $C$ is real.

In order to be in the one cut case we should have that $Q$ has two simple
zeros $z_1$, $z_2$ and therefore one double zero $z_0$. Thus
\begin{align}
	Q(z) = -\frac{1}{4}(z-z_1)(z-z_2)(z-z_0)^2. \label{Q3sym} \end{align} 
Because of the symmetry in the imaginary axis,  we can describe the zeros
with three real parameters $a,b, c$ as follows\footnote{There is an a priori possibility of $Q$
having a double zero and two simple zeros on the imaginary axis. 
However, it turns out that such an ansatz will lead to a contradiction. }
\begin{align} \label{Q3zeros}
	z_0=-ai, \qquad z_1=-b+ci, \qquad	z_2=b+ci,
\end{align}
with $b > 0$.

\begin{Lem} \label{LQ1} 
Suppose that $Q$ is a polynomial \eqref{Q3expansion}, \eqref{Q3sym} 
with zeros \eqref{Q3zeros}.
Then $b$ is the unique positive real root of the equation 
\begin{align} 
	b^6 - 2Kb^4 -8=0
		\label{b_vs_K} \end{align}
		and $a$ and $c$ satisfy
\begin{align} \label{ac_vs_b}
a=c=\frac{2}{b^2}.
\end{align}
The unknown parameter $C$ in \eqref{Q3expansion} is given by 
\begin{equation} \label{C_vs_b}
	C=-\frac{b^6+4}{b^8}. 
	\end{equation}
\end{Lem}
\begin{proof} 
Substituting the values \eqref{Q3zeros} into \eqref{Q3sym} and equating coefficients
with \eqref{Q3expansion} leads to the system of algebraic equations
\begin{equation} \label{Q3equations}
\begin{cases}
\frac{i}{2}(c-a)=0, \\
\frac{1}{4}b^2+\frac{1}{4}c^2-ca+\frac{1}{4}a^2 = \frac{1}{2}K, \\
\frac{ia}{2}(b^2+c^2-ca)=i, \\
-\frac{1}{4} (b^2 + c^2) a^2 = C.
\end{cases}
\end{equation}
Thus $a=c$ and the third equation reduces to $a b^2 = 2$. This proves \eqref{ac_vs_b}.

Plugging \eqref{ac_vs_b} into the second equation of \eqref{Q3equations}
leads to 
\[ \frac{1}{4}b^2-\frac{1}{2}a^2 - \frac{1}{2}K=0 \]
which can be rewritten as \eqref{b_vs_K}.
The value of $C$ as given in \eqref{C_vs_b} then follows from the fourth
equation of \eqref{Q3equations}.

Finally, we note that for a given real $K$, the equation  \eqref{b_vs_K} has exactly one positive real zero,
as follows for example from the Descartes rule of signs (see part 5, chapter 1 from \cite{PS}), 
according to which the number of
positive zeros of  a polynomial with real coefficients  is either equal to the number of sign 
changes among consecutive nonzero coefficients, or is less than it by a multiple of two.
\end{proof}

\subsection{Lemma on the trajectories of $-Q(z)dz^2$}

Lemma \ref{LQ1} shows that for any given $K \in \mathbb R$ there is only one
possible candidate for $Q$ with two simple zeros and one double zero 
satisfying \eqref{Q3zeros}. It will give the
quadratic differential $-Q(z) dz^2$ and in order to have a single arc,
we need to have that one of the three critical trajectories that starts at $z_1$ ends at $z_2$. 

\begin{Lem} \label{LQ1_candidate}
Let $K \in \mathbb R$ and let $Q$ be the polynomial \eqref{Q3sym} associated with $K$ as
described above. Then there is a critical value $K^*$ such that the following hold.
\begin{enumerate}
\item[\rm (a)] For $K < K^*$, there is a  critical trajectory $\gamma$ of the quadratic differential $-Q(z)dz^2$ connecting $z_1$ and $z_2$. 
\item[\rm (b)] For $K = K^*$, there is a critical trajectory $\gamma_1$ connecting $z_1$ and the double zero $z_0=-ia$
on the imaginary axis, and a critical trajectory $\gamma_2$ connecting $z_2$ and $z_0$. The trajectories
$\gamma_1$ and $\gamma_2$ meet at $z_0$ at an angle of $\frac{\pi}{2}$.
\item[\rm (c)] For $K >  K^*$, there is no union of bounded critical trajectories connecting $z_1$ and $z_2$.  
\end{enumerate}
\end{Lem}
\begin{proof} 
We consider the function 
\begin{align} \nonumber
	D(z) & = \frac{1}{\pi i}\int_{z_1}^z Q^{1/2}(s)ds \\
		& = \frac{1}{2\pi}\int_{z_1}^z(s-z_0)\left[(s-z_1)(s-z_2)\right]^{1/2}ds \label{D3}
\end{align}
where the branch of the square root is specified below.
The critical trajectories emanating from $z_1$ are characterized by $\Im D(z) = 0$.
We want to show that one critical trajectory comes to the imaginary axis and therefore we
study
\begin{equation} \label{Fy} 
	F(y) = \Im D(iy), \qquad y \in \mathbb R. 
	\end{equation}

For a given $y \in \mathbb R$ we choose the branch cut of	$\left[(s-z_1)(s-z_2)\right]^{1/2}$
along a path from $z_1$ to $z_2$ that intersects the imaginary axis once in a number $iy^*$ with
$y^* < y$. In addition we take $\left[(s-z_1)(s-z_2)\right]^{1/2} \sim s$ as $s \to \infty$.
Then \eqref{D3} can be evaluated. We use \eqref{Q3zeros} and change
variables $s = u + ic$ to obtain
\begin{align*} 
	D(z) & = \frac{1}{2\pi} \int_{-b}^{z-ic} (u + (a+c)i) \sqrt{u^2-b^2} du \\
	& = \frac{1}{6\pi} ((z-ia)^2 - b^2)^{3/2} - \frac{a}{2\pi i} (z-ia) ((z-ia)^2 - b^2)^{1/2} \\
	& \quad + \frac{1}{\pi i} \log (z-ia + ((z-ia)^2 - b^2)^{1/2}) - \frac{1}{2\pi} \log(-b)
		\end{align*}
where we also used that $c=a$ and $ab^2 =2$, see \eqref{ac_vs_b}.

Then by \eqref{Fy} and the choice of the branch cut,
\begin{align} \nonumber
   F(y) & = -\frac{1}{6\pi} ((y-a)^2 + b^2)^{3/2}  - \frac{a}{2\pi} (y-a) ((y-a)^2 + b^2)^{1/2} \\
   & \quad - \frac{1}{\pi} \log (y-a + ((y-a)^2 + b^2)^{1/2}) 
		+ \frac{1}{\pi} \log(b), \label{FoniR}
		\end{align}
where the fractional powers are all non-negative. 

The derivative of \eqref{FoniR} is 
\begin{align*} 
	F'(y) & = - \frac{1}{2\pi} \frac{y^3- a y^2-a^2y+b^2y+a^3+2}{\sqrt{(y-a)^2+b^2}} \\
	& = - \frac{1}{2\pi} (y+a)\sqrt{(y-a)^2+b^2},
\end{align*}
which is zero for $y=-a$ only. The derivative is negative for $y > -a$ and positive
for $y < -a$. Thus $F$ is strictly decreasing for $y > -a$ and strictly increasing
for $y < -a$ (up to the branch cut). Also
\[ F(y) \to - \infty, \qquad \text{as } y \to + \infty. \]

Thus $F$ has a zero if and only if $F(-a) \geq 0$. We  compute from \eqref{FoniR}
\begin{align*}  
	F(-a) & = - \frac{1}{6\pi} (4a^2 + b^2)^{3/2} + \frac{a^2}{\pi} (4a^2 + b^2)^{1/2} 
	- \frac{1}{\pi} \log(-2a + \sqrt{4a^2 + b^2}) + \frac{\log b}{\pi},
	\end{align*}	
and using again $a b^2 =2$, see \eqref{ac_vs_b},
\begin{align}
	F(-a)  =	 \frac{1}{6 \pi} \left[ (2a^{3/2} - 2 a^{-3/2}) \sqrt{4a^3 + 2} +  6 \log(2 a^{3/2} + \sqrt{4a^3 +2}) 
		- 3 \log 2 \right] 
	 \label{eq_determine_a_crit}. 
\end{align}
The derivative of \eqref{eq_determine_a_crit} with respect to $a$ is remarkably simple
\begin{align} \label{dFda}
	\frac{d}{d a} F(-a) &  = \frac{1}{\sqrt{2} \pi} \frac{(2 a^3 + 1)^{3/2}}{a^{5/2}},
\end{align}
which is positive, since $a > 0$. Thus $a \mapsto F(-a)$ is strictly increasing as a function of $a > 0$.
For $a=1$, we have
\[ F(-a) = F(-1) = - \log (-2 + \sqrt{6}) + \frac{1}{2} \log 2 \approx 1.1462 > 0 \]
and $F(-a) \to -\infty$ as $a \to 0+$.
Thus there is a unique value $a = a^* > 0$ such that $F(-a^*) = 0$, namely
\[ a^* = 0.6821733958\cdots \]
Since $ab^2 = 2$ the corresponding value of $b > 0$ is
\[ b^*  = 1.712251710\cdots \]
and according to \eqref{b_vs_K}
\begin{equation} \label{valueKstar} 
	K^* = \frac{1}{2} (b^*)^2 - \frac{4}{(b^*)^4} = \frac{1}{a^*} - (a^*)^2 \approx 1.0005424 \cdots. 
	\end{equation}
It is easy to see from \eqref{b_vs_K} and \eqref{ac_vs_b} that $K < K^*$ if and only if $a > a^*$ and
$K > K^*$ if and only if $0 < a < a^*$.

\medskip

For $K > K^*$ we have $a < a^*$ and then $F(y) < 0$ for every $y \in \mathbb R$, which means
that the level line $\Im D(z) = 0$ 
does not intersect the imaginary axis. 
Hence the critical trajectories that emanate from $z_1$ do not intersect the imaginary
axis, and so there is no critical trajectory (or union of critical trajectories) that connects $z_1$ with $z_2$. 
This proves part (c) of the lemma. 

\medskip

Now let $K < K^*$, so that $a > a^*$. There are three critical trajectories emanating from $z_1$. Suppose
none of them ends at $z_2$. They also cannot end at $z_0$ since $F(-a) > 0$ and $z=z_0 = ia$ is not
on the level line $\Im D(z) = 0$. Then all three critical
trajectories have to stay in the left half-plane. Since trajectories cannot be closed, the
trajectories then tend to infinity at distinct angles   $2 \pi/3$, $\pi$ and $4 \pi/3$, 
see Lemma \ref{lem:angles} and \eqref{angles2}.
By symmetry the three trajectories emanating from $z_2$ stay in the right half-plane and
tend to infinity at distinct angles $-\pi/3$, $0$ and $\pi/3$. 

Now since $a > a^*$ there are two values $iy$, say $iy_1$ and $iy_2$ with $y_1 > y_2$, 
on the imaginary axis with $F(y) = \Im D(iy) = 0$. 
This means that there are two additional level lines $\Im D(z) = 0$ that cross the imaginary axis. 
These are non-critical trajectories of the quadratic differential that therefore extend
to infinity in both directions. However, each of the admissible directions is already
taken by a critical trajectory with $\Im D(z) = 0$, and we find a contradiction because
of Lemma \ref{differentD}.  As a consequence there has to be a critical trajectory connecting $z_1$ and $z_2$
that intersects the imaginary axis in one of the points $iy_1$ or $i y_2$.
This proves part (a) of the lemma.

\medskip
For $K = K^*$ we have $a=a^*$ and then the double zero $z_0$ is part of the level curve
$\Im D(z) = 0$. It is in fact the only point on the imaginary axis on this level curve. 
By a continuity argument from $K < K^*$ the critical trajectory that connects
$z_1$ with $z_2$ will pass through $z_0$. It makes an angle $\frac{\pi}{2}$ because 
of the local structure of trajectories at a double zero, see Lemma \ref{m_traject},
and the symmetry in the imaginary axis.
\end{proof}

From the proof of Lemma \ref{LQ1_candidate} we have that the
critical value $a^*$ is determined as $F(-a^*)$, with $F(-a)$ given by \eqref{eq_determine_a_crit}.
Putting $v = a^{-1/3}$ it is then easy to see that $v^* = (a^*)^{-1/3}$ 
satisfies \eqref{vequation}. 
Because of \eqref{valueKstar} we then have
\[ K^* = (v^*)^{-1/3} - (v^*)^{-2/3} \]
and so $K^*$ defined in \eqref{valueKstar} agrees with the definition
given in \eqref{defKstar} in Theorem \ref{Thm_single_cut_V1}.

\subsection{Proof of Theorem \ref{Thm_single_cut_V1} (c)}
Part (c) of Lemma \ref{LQ1_candidate} is enough to prove part (c) of
Theorem \ref{Thm_single_cut_V1}.

\begin{proof} 
Suppose $K > K^*$. If the zeros would accumulate on one arc, then this arc
would be a critical trajectory of the quadratic differential $-Q(z)^2 dz$ with
$Q$ given by \eqref{Q3expansion}--\eqref{Q3zeros}. The critical trajectory connects the two simple
zeros of $Q$.  If $K > K^*$ then $a < a^*$ and by item (c) 
of Lemma \ref{LQ1_candidate} there is no critical trajectory between the two
zeros $z_1$ to $z_2$ of $Q$.  This contradiction shows that the zeros do not
accumulate on one arc if $K > K^*$.

They cannot accumulate on more than two arcs since each endpoint would be a
zero of $Q$ given by \eqref{defQ}. However, $Q$ has degree four, and so cannot
have more than four zeros, which means at most two arcs.
\end{proof}

\subsection{Lemma on analytic extensions}

Next we focus on the case $K < K^*$.
To prove part (a)  of Theorem \ref{Thm_single_cut_V1} we need to know
that the critical trajectory $\gamma$ from $z_1$ and $z_2$ 
(which exists because of part (a) of Lemma \ref{LQ1_candidate})
extends to a contour $\Gamma$ in the class $\mathcal T_{2,1}$ which is 
an $S$-curve in the external field $\Re V$.

\begin{Lem} \label{extension}
\begin{enumerate}
\item[(a)]
Let $K < K^*$. Then the critical trajectory $\gamma$ from $z_1$ to $z_2$ has an
extension to a curve
\[ \Gamma = \Gamma_1 \cup \gamma \cup \Gamma_2 \]
where for $i=1,2$,  $\Gamma_i$ is a critical trajectory of $Q(z) dz^2$
that extends from $z_i$ to infinity.
The curve $\Gamma$ belongs to the class $\mathcal T_{2,1}$.

\item[(b)] Let $K = K^*$. Then the same statement holds true for the
 union $\gamma_1 \cup \gamma_2$ of critical
trajectories (see part (b) of Lemma \ref{LQ1_candidate}).
\end{enumerate}
\end{Lem}

\begin{proof}
Assume $K < K^*$.
In the proof we continue to use the notions from the proof of Lemma \ref{LQ1_candidate}.
That is, $F(y) = \Im D(iy)$ (defined in \eqref{Fy})  takes its maximum value at $y=-a$,
it is strictly decreasing for $y > -a$   and strictly increasing for $y < -a$
with $\lim_{y \to \pm \infty} F(y) = -\infty$.
Since $K < K^*$ we have $F(-a) > 0$ and there are two values $y_1$, $y_2$ with
$y_1 > -a > y_2$ such that $F(y_1) = F(y_2) = 0$.
We can also check from  \eqref{FoniR} that $F(a) < 0$ so that
in fact
\begin{equation} \label{yordering} 
	y_2 < -a < y_1 < a. 
	\end{equation}

The critical trajectory from $z_1$ to $z_2$ intersects the imaginary axis in either $iy_1$ or $iy_2$.
Suppose that it does so in $iy_2$.
Since trajectories do not intersect we will then have that the trajectory through $iy_1$ is above
the critical trajectory, which goes from $z_1 = -b + ia$ to $z_2 = b + ia$. (Recall that $\Im z_{1,2} = c = a$, see \eqref{ac_vs_b}.) 
By \eqref{yordering}
this means that the trajectory  passing through $iy_1$ intersects the horizontal
segment 
\[ [z_1, z_2] := \{ x + ia \mid -b < x < b \} \]
between $z_1$ and $z_2$, and so $\Im D(z) = 0$ somewhere in the interior of this segment. 
However, for $z = x+ia$ with $-b < x < b$ we have by \eqref{D3}
\begin{align*} 
	D(x+ia) & = \frac{1}{2\pi} \int_{-b}^{x} (s  + 2ia)  (s^2 - b^2)^{1/2} ds \\
		& = \frac{i}{2\pi} \int_{-b}^x (s+2ia) \sqrt{b^2-s^2} ds 
	\end{align*}
and therefore
\[ \Im D(x+ia) = \frac{1}{2\pi} \int_{-b}^x s \sqrt{b^2-s^2} ds, \qquad  - b < x < b,  \]
which is $< 0$.  Thus the trajectory $\Im D(z) = 0$ passing through $iy_1$ cannot intersect the
segment $[z_1,z_2]$, and it follows that the critical trajectory does not  intersect the
imaginary axis in $iy_2$.
Therefore it intersects the imaginary axis in $iy_1$. The trajectory through $iy_2$ then tends to
infinity and it has to stay below the critical trajectories that emanate from $z_1$ and $z_2$.
It means that it tends to infinity at angles $-\pi/3$ and $-2\pi/3$. The unbounded critical trajectories
tend to infinity at the other angles $0$, $\pi/3$, $2\pi/3$ and $\pi$ and we have the situation
as sketched in Figure \ref{figure:subcrit_dif}.

At the simple zeros $z_1$ and $z_2$ the critical trajectory $\gamma$ extends analytically to
vertical trajectories $\Gamma_1$ and $\Gamma_2$, that is, to trajectories of the quadratic 
differential $Q(z) dz^2$. On $\Gamma_1$ and $\Gamma_2$ we have that $\Re D(z)$ is constant
and $\Im D$ is monotonically increasing or monotonically decreasing. Then $\Gamma_1$ and $\Gamma_2$
intersect with the level lines $\Im D = 0$ only at $z_1$ and $z_2$.  
From the global picture in Figure \ref{figure:subcrit_dif} it is then clear
that these orthogonal trajectories end at angles $\pi/6$ and $5 \pi/6$.
It means that 
\[ \Gamma = \Gamma_1 \cup \gamma \cup \Gamma_2 \]
is an analytic contour in $\mathcal T_{2,1}$ in case $K < K^*$. This proves part (a).

In the limit when $K \nearrow K^*$, we have that $y_1 \to -a$ and $y_2 \to -a$.
Then four trajectories are emanating from $z_0 = -ia$. One that is connected with
$z_1$, one with $z_2$, and two unbounded ones that end at angles $-\pi/3$ and $-2\pi/3$. 
The unbounded critical trajectories that emanate from $z_1$ and $z_2$ still end at
the other remaining directions at infinity, and then we can argue as in the case $K < K^*$
to conclude that $\Gamma$ is an extension in $\mathcal T_{2,1}$,
and part (b) of the lemma follows.
\end{proof}

\begin{figure}[!t]
\centering
\begin{overpic}[scale=.35,unit=1mm]%
     {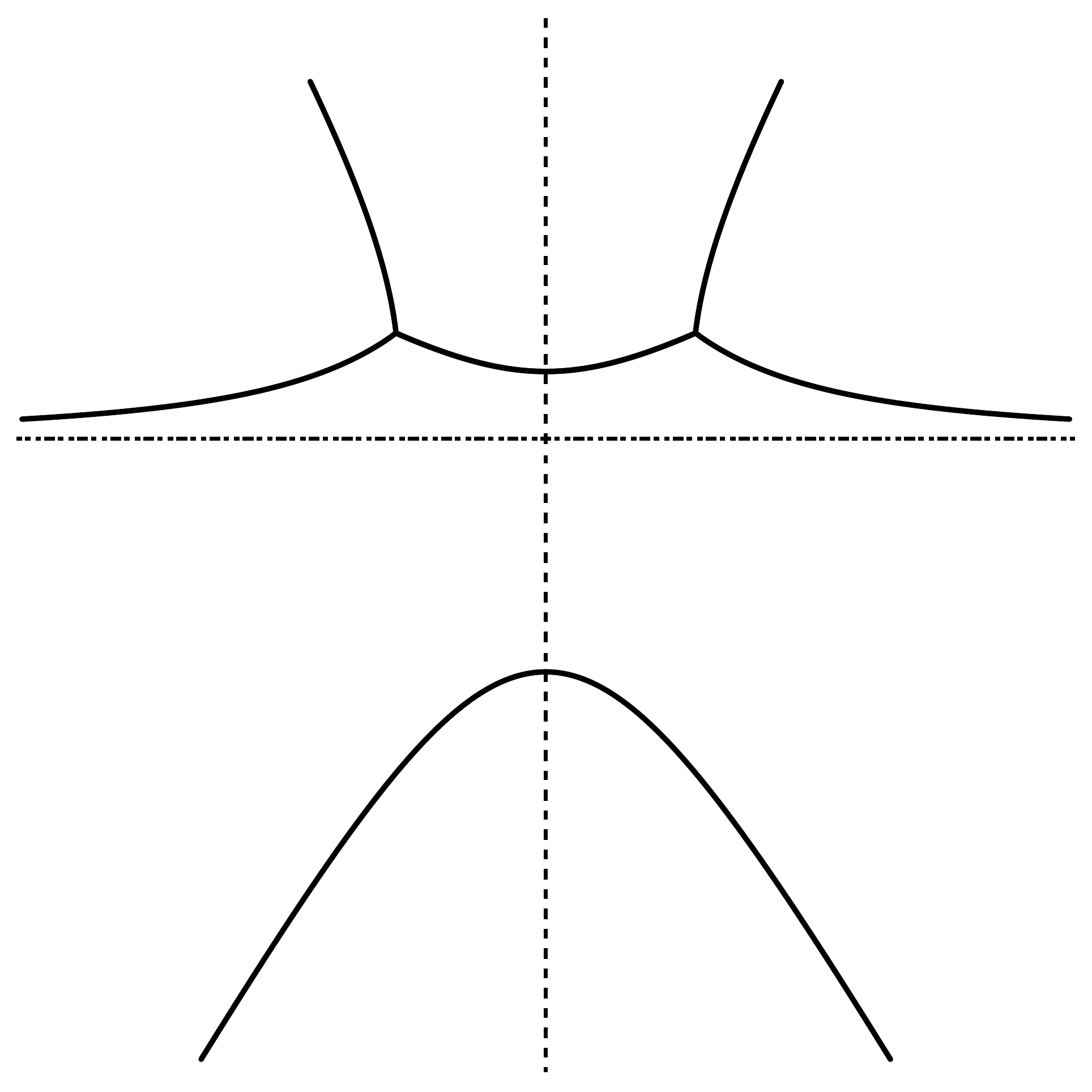}
\put(28,70){ \footnotesize $z_1$}
\put(65,70){ \footnotesize $z_2$}
\put(50.5,68){\footnotesize  $iy_1$}
\put(50.5,40){\footnotesize  $iy_2$}
  \end{overpic}
\caption{The configuration of zero level curves for $\Im D$, for subcritical $K<K^*.$\label{figure:subcrit_dif} }
\end{figure}
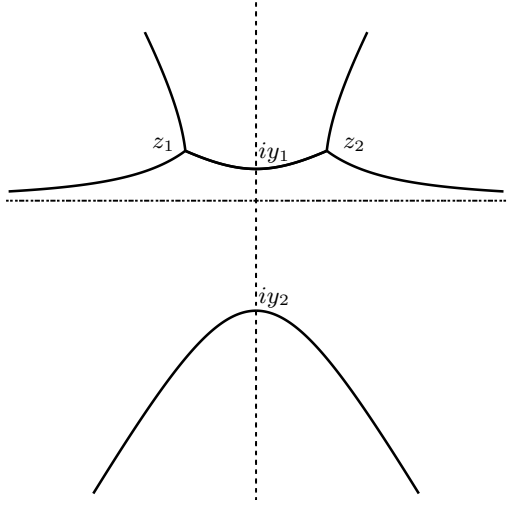

\subsection{Proof of Theorem \ref{Thm_single_cut_V1} (a) and (b)}

Now we can complete the proof of Theorem \ref{Thm_single_cut_V1}.

\begin{proof}
Let $K \leq K^*$. Then by part (a) of Lemma \ref{LQ1_candidate} there is
a critical trajectory $\gamma$ 
(or a union $\gamma = \gamma_1 \cup \gamma_2$ of two critical trajectories in case $K = K^*$) 
of $-Q(z)dz^2$ that connects $z_1$ and $z_2$.

By Lemma \ref{extension} this trajectory (or union of two trajectories) extends to an unbounded contour
\[ \Gamma = \Gamma_1 \cup \gamma \cup \gamma_2 \]
in the class $\mathcal T_{2,1}$, where $\Gamma_1$ and $\Gamma_2$ are unbounded orthogonal trajectories.
It remains to show that $\Gamma$ has the $S$-property in the external field $\Re V$
and $\supp(\mu_{\Gamma}) = \gamma$, since then by Theorem \ref{theoremGR} 
the zeros of $P_n$ accumulate on $\gamma$ as $n \to \infty$.

The proof that $\Gamma$ has the $S$-property was done for $K = 0$ in \cite[Theorem 2.2]{DHK}.
There it was also shown that
\[ d\mu_{\Gamma}(s) = \frac{1}{\pi i} Q^{1/2}_+(s) ds, \qquad s \in \gamma \]
and $\supp(\mu_{\Gamma}) = \gamma$. The same proof works for general $K \leq K^*$. 
This completes the proof of Theorem \ref{Thm_single_cut_V1}.
\end{proof}

\section{Proof of Theorem \ref{Thm_single_cut_V2}\label{sect_deg5}}

\subsection{Introduction}

We are now considering the quintic case
\[
 V(z) = -\frac{i z^5}{5}.
\]
As mentioned in Section \ref{sect_intro}, for this potential there are two possible combinations 
of sectors that respect the symmetry with respect to the imaginary axis. They correspond 
to two sets of curves, $\mathcal T_{3,1}$ and $\mathcal T_{4,5}$. Example contours were illustrated 
in Figure \ref{Fig_admitted_sect}(b) and (c).

The reasoning for the quintic case is similar to that of the cubic case in \S\ref{sect_deg3}, 
though somewhat more involved. We start by determining two possible candidates for $Q$, labelled 
$Q_1$ and $Q_2$. We intend to show that they correspond precisely to the two cases illustrated 
in Figure \ref{Fig_admitted_sect}. To that end, for each polynomial $Q_p$, $p = 1,2$, we show that 
there is a critical trajectory $\gamma_p$ connecting its simple zeros. Next, we show the existence 
of analytic extensions $\Gamma_p$ for both curves that tend to infinity in the right sectors. The 
proof is finalized by showing that the resulting global curves have the $S$-property in the
external field $\Re V$.

\subsection{Candidates for the polynomial $Q$}

In order to be in the one cut case, the polynomial $Q$ should satisfy the following conditions:
\begin{itemize}
\item $\deg Q=8$.
\item $Q$ is symmetric with respect to the imaginary axis: $Q(z)=\overline{Q(-\bar{z})}$.
\item The asymptotic behavior of $Q$ at infinity is as specified in \eqref{Qasympt}, which implies that $Q$ is of the form:
\begin{equation}
Q(z)=-\frac{1}{4}z^8+iz^3+\mathcal{O}(z^2).\label{Q_expansion}
\end{equation}
\item $Q$ has exactly two simple zeros $z_1$ and $z_2$ and three double zeros $z_0,z_3$ and $z_4$, i.e.,
\begin{equation}
 Q(z)=-\frac{1}{4}(z-z_1)(z-z_2)(z-z_0)^2(z-z_3)^2(z-z_4)^2.\label{Q_sym}
\end{equation}
\end{itemize}
The symmetry implies that the zeros of $Q$ are on the imaginary axis or appear in pairs: $(z,-\bar{z})$. 
Thus, at least one of the double zeros has to be on the imaginary axis. We assume that there is exactly 
one and we denote it by $z_0$. We also assume that the single zeros appear as a pair, rather than as 
two separate zeros on the imaginary axis. With these assumptions, we can describe the symmetry of all 
zeros using five real parameters $a,b,c,d$ and $e$, with $b$ and $d$ positive:
\begin{equation}\label{Q_zeros}
z_{0}=-ai, \quad z_{1}=-b+ci, \quad z_{2}=b+ci, \quad z_{3}=-d+ei, \quad z_{4}=d+ei,
\end{equation}
It can be verified with computations similar to the following that other assumptions do not lead to 
a solution for the polynomial $Q$.

We proceed by determining all polynomials $Q$ that satisfy these conditions. As in the cubic case, 
we find the parameters $a,b,c,d,e$ as the solution of a set of equations. The equations are non-linear 
in this case, but they still consist of polynomials in the parameters. This admits the use of a concept
 from linear algebra, the resultant, which leads to analytic expressions for the parameters. 

\begin{Thm}\label{Res_silvester}
Suppose that $f(z) = a_0z^m+\cdots +a_m$ and $g(z) = b_0z^n+\cdots +b_n$ are two complex polynomials 
so that $a_0 \neq 0 \neq  b_0$. Then the following two properties are equivalent:
\begin{enumerate}
\item $f$ and $g$ share a common root.
\item $R_{f,g} = 0$, where
\begin{equation}
R_{f,g}=\det \begin{pmatrix}a_0&a_1&\cdots &a_m& 0&0&\cdots &0\\
0&a_0&\cdots &&a_m&0&\cdots &0\\
0&0&a_0&\cdots &&a_m&\cdots &0\\
& & & \ddots &&&\ddots &\\
& & & &a_0 &&&a_m\\
b_0&b_1&&\cdots &&b_n& \cdots &0\\
&\ddots &&&&&\ddots &\\
0&&b_0&\cdots &&&\cdots &b_n\\
\end{pmatrix}.
\end{equation}
\end{enumerate}
\end{Thm}
\begin{proof} 
See e.g.\ Theorem 4.1 of \cite{BS}.
\end{proof}

\begin{Lem}\label{LQ2}
There are two polynomials of degree $8$ that satisfy conditions \eqref{Q_expansion} and \eqref{Q_sym} 
with zeros of the form \eqref{Q_zeros}. 
Their zeros are specified by the parameters $\{a_p,b_p,c_p,d_p,e_p\}$, with values for $p=1$:

\begin{align*}
a_1&=\frac{1}{3}\left(-A_1-\frac{4+\sqrt{30}}{A_1}+1\right) \left(-\frac{27}{14}+5\frac{\sqrt{30}}{14}\right)^{1/5}\approx -1.1082,\\
b_1&=\frac{1}{42}\sqrt{36+6\sqrt{30}}B_1 \approx 1.3489, \\
c_1&=\left(-\frac{27}{14}+5\frac{\sqrt{30}}{14}\right)^{1/5} \approx 0.4877, \\
d_1&=\frac{B_1}{28\sqrt{3}}\sqrt{A_1^2+(8+\sqrt{30})+\frac{\sqrt{1410\sqrt{30}+7740}(4-\sqrt{30})}{14}A_1-2(4+\sqrt{30})}\approx 0.6781,\\
e_1&=\frac{1}{6}\left(-A_1-\frac{4+\sqrt{30}}{A_1}-2\right) \left(-\frac{27}{14}+5\frac{\sqrt{30}}{14}\right)^{1/5}\approx -0.7979,
\end{align*}
where the constants $A_1$ and $B_1$ are given by
\[
 A_1=\left(12\sqrt{30}+62+\sqrt{1410\sqrt{30}+7740}\right)^{\frac{1}{3}}\quad\mbox{and} \quad B_1=\left(-1037232+192080\sqrt{30}\right)^{\frac{1}{5}}.
\]
For $p=2$ the values are
\begin{align*}
a_2&=-\frac{1}{3}\left(-A_2+\frac{4-\sqrt{30}}{A_2}+1\right)\left(\frac{27}{14}+5\frac{\sqrt{30}}{14}\right)^{1/5} \approx -0.9820, \\
b_2&=\frac{1}{42}\sqrt{36-6\sqrt{30}}B_2  \approx  0.7744,\\
c_2&=-\left(\frac{27}{14}+5\frac{\sqrt{30}}{14}\right)^{1/5}\approx -1.3118\\
d_2&=\frac{B_2}{28\sqrt{3}}\sqrt{A_2^2+(8-\sqrt{30})+\frac{\sqrt{-1410\sqrt{30}+7740}(4+\sqrt{30})}{14}A_2-2(4-\sqrt{30})}\approx  1.0344.\\
e_2&=-\frac{1}{6}\left(-A_2+\frac{4-\sqrt{30}}{A_2}-2\right)\left(\frac{27}{14}+5\frac{\sqrt{30}}{14}\right)^{1/5}\approx  0.1649,
\end{align*}
with the constants $A_2$ and $B_2$ given by
\[
A_2=\left(-12\sqrt{30}+62+\sqrt{-1410\sqrt{30}+7740}\right)^{\frac{1}{3}} \quad \mbox{and} \quad B_2=\left(1037232+192080\sqrt{30}\right)^{\frac{1}{5}}.
\]
\end{Lem}
\begin{proof}
The proof of the lemma consists of three steps. First, we obtain a system of equations that has to be 
satisfied by the parameters $a,b,c,d,e$ by matching the parameterized polynomial \eqref{Q_sym} with the 
asymptotic formula \eqref{Qasympt}. Next, the resulting nonlinear system is simplified to just two equations 
in two unknowns by eliminating parameters. Finally, Theorem \ref{Res_silvester} is applied to find a third 
equation, from which the two families of parameters can be explicitly found.

Thus, we start by matching the coefficients of a polynomial of the form \eqref{Q_sym}, paramaterized by 
\eqref{Q_zeros}, to the coefficients of the asymptotic formula \eqref{Qasympt}. The leading order coefficient 
is matched by construction. For degree $7$, the equation is rather simple:
\[
\frac{i}{2}(-a+c+2e)=0,
\]
from which we find $e$ as
\begin{equation}\label{eq_e}
e=\frac{1}{2}(a-c).
\end{equation}
Taking this expression into account, matching the coefficients of degree $6$ down to $3$ yields the system of 
equations:
\begin{equation}\label{sys1}
\left\{ \begin{array}{l}
-3a^2+2ac+2b^2-3c^2+4d^2=0,\\
-ac^2+a^2c+4b^2c+4ad^2-4cd^2+a^3-c^3=0,\\
-16ab^2c+16acd^2-38a^2c^2+20a^3c+24a^2b^2+20ac^3 \\
\qquad +24a^2d^2-9a^4-9c^4-16d^4-32b^2d^2-24b^2c^2+24c^2d^2=0, \\
-32+2a^2c^3+3a^5+20a^2b^2c-24a^2cd^2+a^4c-16b^2cd^2+8a^3d^2+16cd^4 \\
\qquad -3c^5-8c^3d^2-4a^3b^2-16ad^4-ac^4-4b^2c^3-2a^3c^2 \\
\qquad -12ab^2c^2+24ac^2d^2-16ab^2d^2=0.
        \end{array}
\right.
\end{equation}

Next, note that $b$ and $d$ only appear squared in these expressions. The unknowns $b^2$ and $d^2$ are 
found in terms of $a$ and $c$ from the first two equations of system \eqref{sys1}:
\begin{equation}\label{eq_bdac}
\left\{ \begin{array}{rl}
          b^2&=\frac{2(c-a)(a^2+c^2)}{3c-a} \\
          d^2&=\frac{7a^2c-5ac^2+a^3+5c^3}{4(3c-a)}
        \end{array}
        \right.
\end{equation}
This reduces the system \eqref{sys1} to
\begin{equation}\label{sys2}
\left\{ \begin{array}{l}
-15c^6+3a^6+36a^3c^3-3a^4c^2-6a^5c-37a^2c^4+30ac^5=0\\
-9c^2+6ac-a^2+3a^6c-3c^7+2ac^6+7a^4c^3-6a^5c^2-4a^3c^4+a^2c^5=0
        \end{array}
\right.
\end{equation}

Finally, we view \eqref{sys2} as two polynomials in $a$, with coefficients depending on $c$. By 
Theorem \ref{Res_silvester}, these polynomials have a common root if and only if the resultant vanishes:
\begin{equation}
c^{12}(28c^{10}+108c^5-3)^3=0.
\end{equation}
If $c=0$ then \eqref{sys2} yields also $a=0$, but \eqref{sys1} becomes inconsistent. We conclude that 
\begin{equation}
28c^{10}+108c^5-3=0.
\end{equation}
This is a quadratic equation in $c^5$ and we find two real solutions
\[
c_1=\left(-\frac{27}{14}+\frac{5}{14}\sqrt{30}\right)^{1/5} \qquad \text{and} 
	\qquad c_2=-\left(\frac{27}{14}+\frac{5}{14}\sqrt{30}\right)^{1/5}.
\]
For these particular values of $c$, the two equations in \eqref{sys2} have a common polynomial 
factor. For $c_1$, this factor is
\begin{equation}
3a^3-3a^2c_1-(3c_1^2+\sqrt{30}c_1^2)a+3c_1^3\sqrt{30}+15c_1^3=0,\label{eq_ac1}
\end{equation}
while for $c_2$ it is
\begin{equation}
3a^3-3a^2c_2-(3c_1^2-\sqrt{30}c_1^2)a-3c_2^3\sqrt{30}+15c_2^3=0. \label{eq_ac2}
\end{equation}
It can be verified that the other cubic polynomials in the factorization of both 
equations of \eqref{sys2} do not share common roots. 
We find the values $a_1$ and $a_2$ as the real roots of the cubic equations \eqref{eq_ac1} and 
\eqref{eq_ac2}. The value of $b$ follows from \eqref{eq_bdac}, while the given value for $d$ 
follows from the original system \eqref{sys1}. 
Finally the value of $e$ follows from \eqref{eq_e}.
\end{proof}

\subsection{Single arc trajectory for the quadratic differential $-Q_p(z)dz^2$}

According to Lemma \ref{LQ2} there are two possible candidates for the polynomial $Q$.
We label these as $Q_1$ and $Q_2$ depending on the values of 
the parameters $a, b, c, d, e$ in \eqref{Q_zeros}.
We show that both $Q_1$ and $Q_2$ give rise to a single arc critical
 trajectory that connects the two simple zeros. These critical trajectories are denoted by $\gamma_p$, 
where $p$ refers to the considered polynomial $Q_p$.
\begin{Lem} For both polynomials $Q_p(z)$, $p=1,2$, the quadratic differential $-Q_p(z)dz^2$ has a critical 
trajectory that connects the two simple zeros $z_1$ and $z_2$ of $Q_p(z).$\label{V2_OneCut}
\end{Lem}
\begin{proof}

The proof is similar for both cases $p=1,2$ and we will give the details only for $p=1$. In both cases, 
the idea is to select a particular critical trajectory $\gamma_p$ emanating from $z_1$ and to prove that it 
stays inside a triangular shaped area. The triangle consists of the horizontal line segment 
connecting $z_1$ and $z_2$ and two titled line segments at angles $\pm \frac{\pi}{4}$. 
One critical trajectory emanating from $z_1$ enters into the triangle and we show that it cannot leave the triangle
through the two sides of the triangle adjacent to $z_1$. Then it has to come to the
imaginary axis, and then by symmetry, it will connect $z_1$ to $z_2$. 
The setup is shown for $p=1$ in Figure \ref{lock_triangle}. 
For the case $p=2$, the slope in the left half-plane is 
positive ($+1$) and the triangle points upwards.

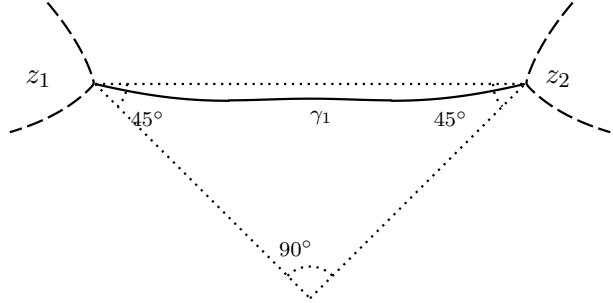
\begin{figure}
\centering
\begin{overpic}
[scale=0.42,unit=1mm]%
{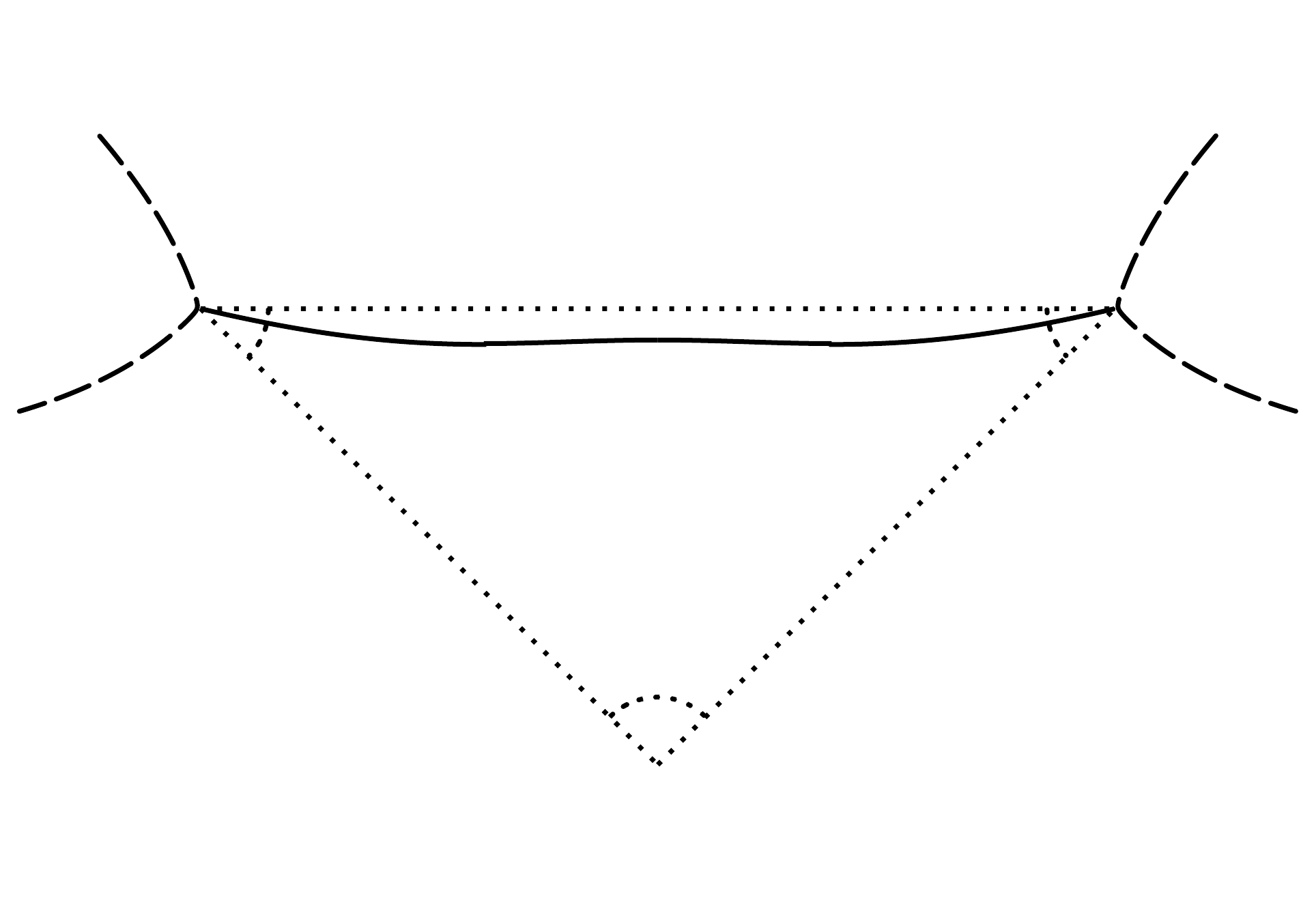}
\put(4,45){\color{black}{$z_1$}}
\put(88,45){\color{black}{$z_2$}}
\put(21,38){\scriptsize{$45^\circ$}}
\put(70,38){\scriptsize{$45^\circ$}}
\put(45,17){\scriptsize{$90^\circ$}}
\put(50,39){\scriptsize{$\gamma_1$}}
\end{overpic}
\caption{The selected critical trajectory $\gamma_1$ (solid) is locked inside a triangle 
(bounded by dotted lines) in the case $p=1$. \label{lock_triangle}}
\end{figure}

First, we show that precisely one critical trajectory emanating from $z_1$ enters into this 
triangular region. From \eqref{dir_crit1}, the angles under which the trajectories leave 
$z_1$ can be calculated exactly. For the case $p=1$ they are, with $k=-0,1,2$ (taken modulo $3$):
\begin{align}\label{angles_above}
	-\frac{1}{3}\arctan\left(\frac{4(\sqrt{30}+3)}{4\sqrt{6}+\sqrt{5}}\right)  + 
	k \frac{2\pi}{3} \quad \textrm{mod} 2\pi \approx
		\begin{cases} -0.1305\pi & \text{ for } k =0, \\
		 0.5362\pi, & \text{ for } k =1, \\
		 -0.7971\pi\, & \text{ for } k = -1.
		\end{cases}
\end{align}
The contour that enters into the triangle corresponds to the choice $k=0$, and this
trajectory is denoted by $\gamma_1$.

We prove that $\gamma_1$ does not cross the sides of the triangle adjacent to $z_1$. 
We recall that $\Im D = 0$ on $\gamma_1$, where 
\begin{equation} \label{Dzfromz1} 
	D(z) = \frac{1}{\pi i} \int_{z_1}^z Q(s)^{1/2} ds, 
	\end{equation}
see \eqref{eq:D}, where we can take any  choice of an analytic square root of $Q$
provided it is continuous along $\gamma_1$.

We have
\begin{equation} \label{Dprime} 
	D'(z) = \frac{1}{\pi i} Q(z)^{1/2}
	= \frac{1}{2\pi} \left[(z-z_1)(z-z_2) \right]^{1/2} (z-z_0)(z-z_3)(z-z_4) 
	\end{equation}
with the $z_j$ as in \eqref{Q_zeros}. Using  $e = \frac{1}{2}(a-c)$ as in \eqref{eq_e}, we then get
\begin{align*} 
	D'(x+ic) & = \frac{1}{2\pi} \sqrt{x^2-b^2} (x+i(c+a))((x+i(c-e))^2-d^2) 
\end{align*}
which implies that for $-b < x < b$,
\begin{align} \label{ImDprime0} 
	\Im D'(x+ic) & = \frac{1}{2\pi} \sqrt{b^2-x^2} 
		\left(x^3  - ( \tfrac{21}{4} c^2 + \tfrac{1}{2} ac - \tfrac{3}{4} a^2 + d^2)x \right).
\end{align}

Now we take the values $a=a_1$, $b=b_1$, etc., given in Lemma \ref{LQ2}.
Then
\[ \Im D'(x+ic_1) \approx \frac{1}{2\pi} \sqrt{b_1^2-x^2}
	\left(x^3 - 0.5171 x \right),   \qquad - b_1 < x < b_1 \]
which has  zeros at $x=0$ and at $x= \pm x^*$ with $x^* \approx 0.72$.
Since $x^* < b_1 \approx 1.34$, these three zeros are in the interval $[-b_1, b_1]$.
It then follows that $x \mapsto \Im D(x+ic)$ is zero at $x=\pm b_1$, takes
a global maximim  at $\pm x^*$ and has
a local minimum at $x=0$. By direct calculation
\[ \Im D(ic) = \frac{1}{2\pi} \int_{b_1}^0 \sqrt{b_1^2-x^2}
	\left(x^3 - 0.5171 x \right) dx  > 0. \]
Thus $\Im D(x+ic) > 0$ for $x \in (-b_1, b_1)$ and it follows that the
critical trajectory $\gamma_1$ does not cross the horizontal interval from $z_1$ 
to $z_2$.

Next, we consider the tilted line segment of the triangle
parametrized by 
\[ z = x + i(c-b-x), \qquad -b < x < 0. \]  
We evaluate $\Re Q$ along this segment with the parameter values
$a = a_1$, $b=b_1$, etc. as given in Lemma \ref{LQ2}. We find that
\begin{multline}
\Re Q(x + i(c_1-b_1-x)) \\
=-4x^8+\frac{8}{21}B_1(3-\sqrt{36+6\sqrt{30}})x^7 -56(c_1-b_1)^3x^5+70(c_1-b_1)^4x^4 \\
	+\left(\frac{429}{1029}+\frac{60}{343}\sqrt{30}+\left(\frac{202}{3087}-\frac{59}{1372}\sqrt{30}\right)
		\sqrt{36+6\sqrt{30}}\right)B_1^2x  \\
	+(2-28(c_1-b_1)^5)x^3 +\frac{5}{21952}(5+\sqrt{30})(-1037232+192080\sqrt{30})^{3/5}, 
\end{multline}
where the constant $B_1$ is as defined in Lemma \ref{LQ2}. This is a polynomial of degree $8$ 
the roots of which can be computed numerically to be approximately
\footnote{The numerical computation of the roots was performed in Maple. Alternatively, 
we may proceed by factoring out the root at $x=-b_1$ analytically from the analytical expression 
for $\Re Q$, and subsequently bound the positive and negative terms in the remaining polynomial of degree $7$ in 
order to show that it has no roots in the interval $(-b_1,0)$. This reasoning avoids all numerical 
computations, yet is omitted for the sake of brevity.}
\begin{align*}
	-2.5741, \, -b_1, \, -0.6044 \pm 0.3452i, \,  -0.1997 \pm 0.3835i, \, 0.3469, \, 1.7393. 
\end{align*}
There are four real roots and none of them are in the interval $(-b_1, 0)$. Since there
are two roots to the right and the leading coefficient is negative, it follows that
\begin{equation} \label{ReQnegative} 
	\Re Q(x+i(c_1-b_1-x)) < 0, \qquad -b_1 < x < 0. 
	\end{equation}

From \eqref{ReQnegative} we obtain, with an appropriate choice for the square root,
\[ \frac{\pi}{4} < \arg Q^{1/2}(x+i(c_1 -b_1-x))  < \frac{3\pi}{4}, \qquad -b_1 < x < 0, \]
and so by \eqref{Dprime}
\[ -\frac{\pi}{4} <   \arg D'(x+i(c_1-b_1-x)) < \frac{\pi}{4}, \qquad -b_1 < x < 0. \]
Then also 
\[ - \frac{\pi}{2} < \arg \left[(1-i) \int_{0}^x D'(s+i(c_1-b_1-s)) \right] < 0, \qquad -b_1 < x < 0. \]
which implies after an integration that 
\begin{align*} 
	D(x+i(c_1-b_1-x)) = (1-i) \int_{0}^x D'(s+i(c_1-b_1-s)) \, ds, \qquad -b_1 < x < 0,
	\end{align*}
is in the lower half-plane, i.e., $\Im D(z) < 0$ for $z = x+i(c_1-b_1-x)$ in  
the tilted line segment.  Thus the critical trajectory $\gamma_1$ cannot escape from this side 
either. 

Therefore, $\gamma_1$ has to come to  the imaginary axis  and by symmetry  connect to $z_2$.
\end{proof}

\subsection{Lemma on analytic extensions}

The next step in the proof consists of showing how the critical trajectories extend to infinity. 
Since $z_1$ and $z_2$ are simple zeros, the trajectory $\gamma_p$ has an analytic extension
to a vertical trajectory of $-Q_p(z) dz^2$.
For the case $p=1$, we show below that the analytic extension of the critical trajectory $\gamma_1$ 
belongs to the class $\mathcal T_{3,1}$. 
Similar arguments show that  the analytic extension of $\gamma_2$ belongs to the class $\mathcal T_{4,5}$.

\begin{Lem}\label{extension_V2}
The critical trajectory $\gamma_1$ from $z_1$ to $z_2$ has an extension to a curve
\[
 \Gamma_1 = \Gamma_{1,1} \cup \gamma_1 \cup \Gamma_{1,2}
\]
where for $i=1,2$, $\Gamma_{1,i}$ is a vertical trajectory of $-Q_1(z) dz^2$ that extends from $z_i$ to infinity.

The curve $\Gamma_1$ belongs to the class $\mathcal T_{3,1}$.
\end{Lem}

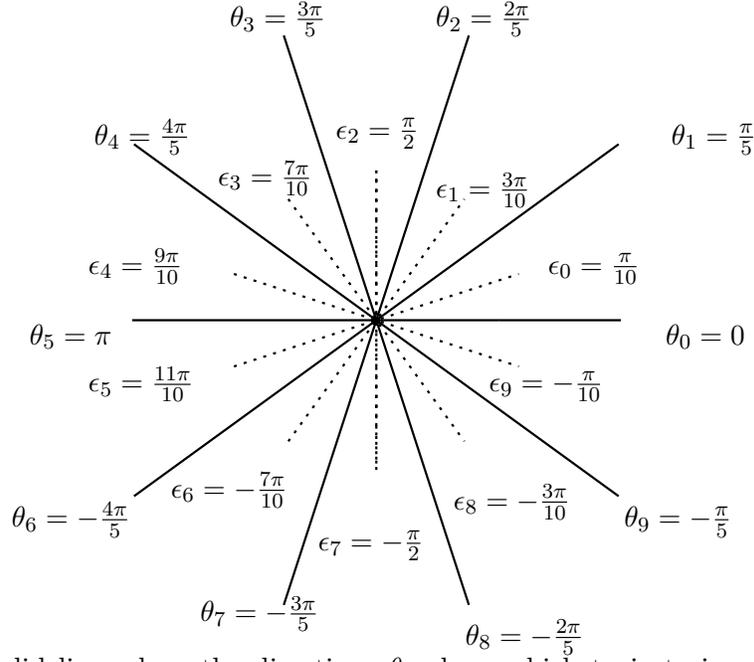
\begin{figure}[!t]
\centering
\begin{overpic}
[scale=0.4,unit=1mm]%
{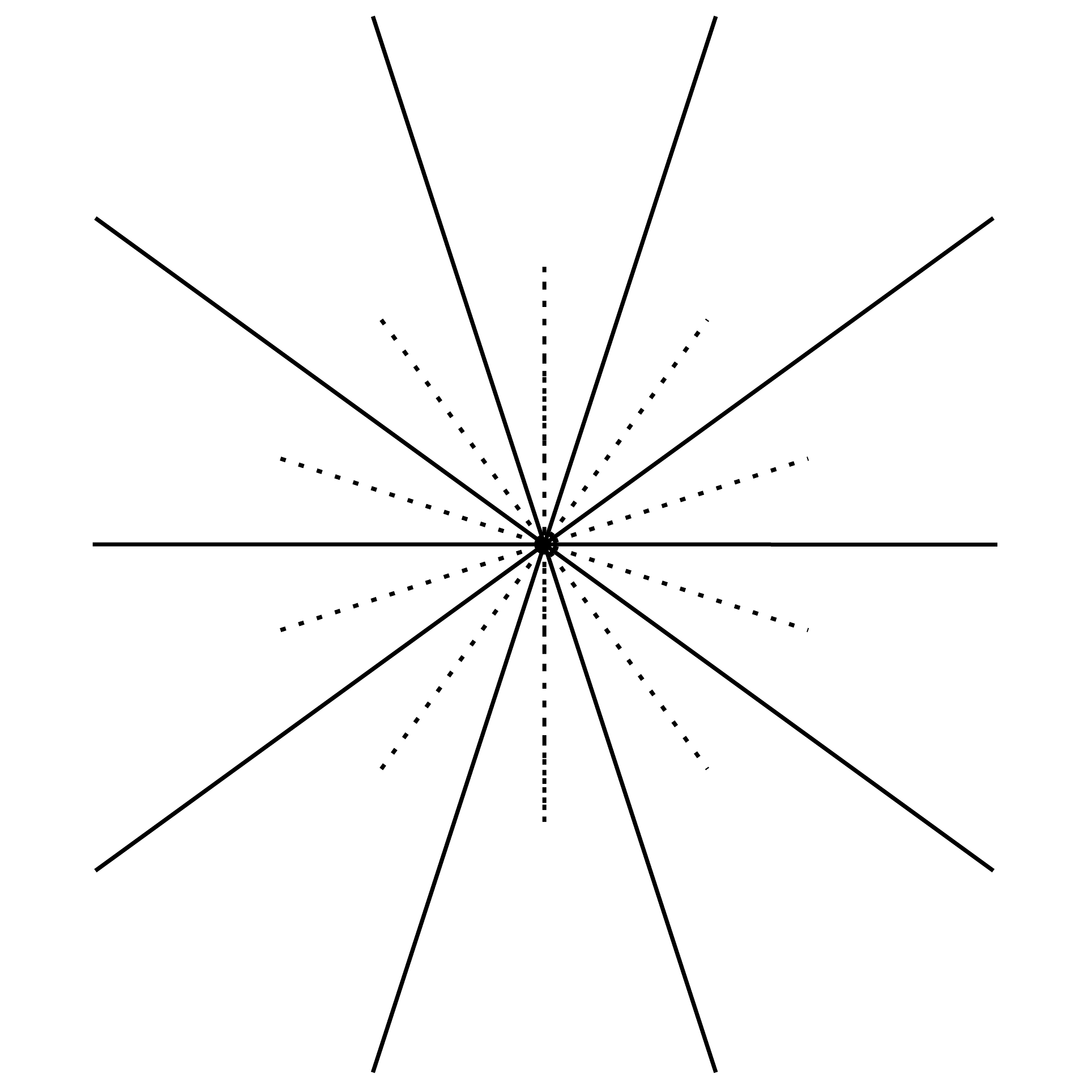}
\put(2,80){$\theta_4=\frac{4\pi}{5}$}
\put(25,100){$\theta_3=\frac{3\pi}{5}$}
\put(60,100){$\theta_2=\frac{2\pi}{5}$}
\put(100,80){$\theta_1=\frac{\pi}{5}$}
\put(99,46){$\theta_0=0$}
\put(-9,46){$\theta_5=\pi$}
\put(-12,15){$\theta_{6}=-\frac{4\pi}{5}$}
\put(20,-1){$\theta_{7}=-\frac{3\pi}{5}$}
\put(65,-5){$\theta_{8}=-\frac{2\pi}{5}$}
\put(92,15){$\theta_{9}=-\frac{\pi}{5}$}
\put(79,58){$\epsilon_0=\frac{\pi}{10}$}
\put(1,58){$\epsilon_4=\frac{9\pi}{10}$}
\put(60,71){$\epsilon_1=\frac{3\pi}{10}$}
\put(43,81){$\epsilon_2=\frac{\pi}{2}$}
\put(23,73){$\epsilon_3=\frac{7\pi}{10}$}
\put(1,38){$\epsilon_{5}=\frac{11\pi}{10}$}
\put(15,20){$\epsilon_{6}=-\frac{7\pi}{10}$}
\put(40,11){$\epsilon_{7}=-\frac{\pi}{2}$}
\put(63,18){$\epsilon_{8}=-\frac{3\pi}{10}$}
\put(69,38){$\epsilon_{9}=-\frac{\pi}{10}$}
\end{overpic}
\caption{The solid lines show the directions $\theta_j$ along which trajectories of $-Q(z)dz^2$ can tend to infinity. 
The dotted lines show the directions $\epsilon_j$ along which vertical trajectories can tend to infinity.
\label{dir_infty}}
\end{figure}

\begin{proof}
Recall from Lemma \ref{lem:angles} that unbounded trajectories of  $-Q_1(z)dz^2$ tend to infinity in one 
of ten asymptotic directions:
\[
\theta_j=\frac{j\pi}{5},\qquad \text{for } j=0,1, \ldots, 9.
\]
These directions are shown in solid lines in Figure \ref{dir_infty}. 
Unbounded vertical trajectories end at infinity under angles
\[
\epsilon_j=\theta_j+\frac{\pi}{10}, \qquad \text{for } j=0,1, \ldots,9.
\]
These directions are plotted in the same figure in dotted lines.

We recall the local structure of the trajectories at $z_1$
as shown in Figure \ref{crit_ext2}. The trajectory $\gamma_1$ starts at $z_1$
at the angle
\[ \psi \approx -0.1305 \pi. \]
The other two trajectories are at angles $\psi \pm \frac{2\pi}{3}$, and they are
labelled $\alpha_1$ and $\beta_1$ as also shown in Figure \ref{crit_ext2}.
The vertical trajectory $\Gamma_{1,1}$ is the analytic continuation of $\gamma_1$ 
and it starts from $z_1$ at an angle
\[ \psi + \pi \approx 0.8695 \pi. \]

If one of $\alpha_1$ and $\beta_1$ would come to the imaginary axis, then, by symmetry, it would continue
to $z_2$, and then together with $\gamma_1$ it would form a closed contour, which is
impossible by Lemma \ref{No_loop}. Thus $\alpha_1$ and $\beta_1$ remain in the left half-plane,
and similarly $\Gamma_{1,1}$ is fully in the left half-plane.

We prove that $\alpha_1$, $\beta_1$, and $\Gamma_{1,1}$ 
do not  intersect the horizontal half line 
\begin{equation} \label{eqL} 
	L = \{z=x+ic_1 \mid x<-b_1\}. 
	\end{equation}

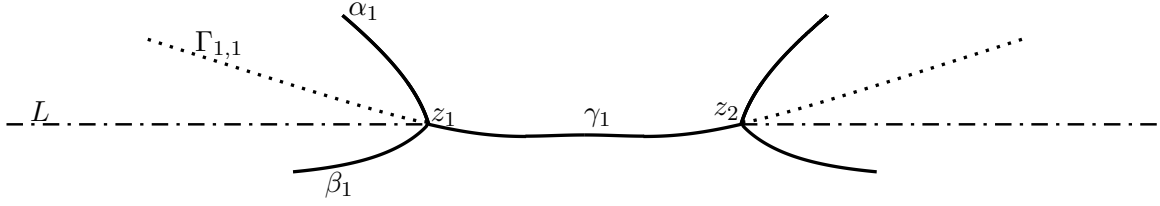
\begin{figure}[t]
\centering
\begin{overpic}[scale=0.8,unit=1mm]%
{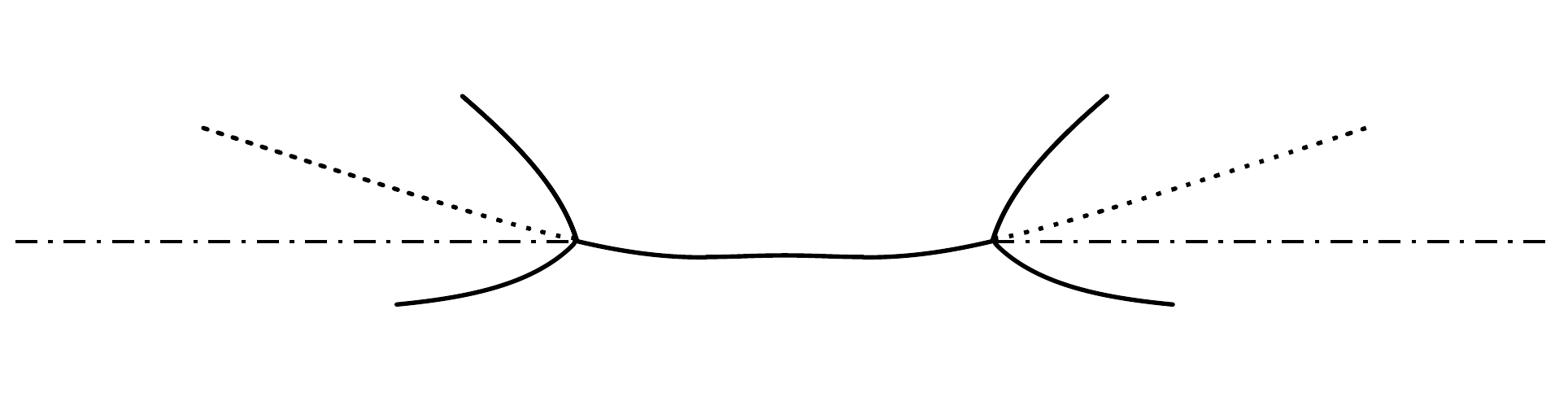}
\put(30,20){$\alpha_1$}
\put(28,5){$\beta_1$}
\put(50,11){$\gamma_1$}
\put(37,11){$z_1$}
\put(61,11.5){$z_2$}
\put(17,17){$\Gamma_{1,1}$}
\put(3,11){$L$}
\end{overpic}
\caption{The figure shows the critical trajectories for the quadratic differential $-Q_1(z)dz^2$ 
(solid line) and the vertical trajectory $\Gamma_{1,1}$. The critical trajectories and the vertical
trajectory do not intersect the horizontal half-line $L$ that is plotted dashed-dotted. \label{crit_ext2} }
\end{figure}

To that end we recall that $\Im D = 0$ on $\alpha_1$ and $\beta_1$ and $\Re D = 0$
on $\Gamma_{1,1}$ where $D$ is given by \eqref{Dzfromz1}. We then have by \eqref{Q_sym} and \eqref{Dzfromz1}
\[ D'(z) = \frac{1}{\pi i} Q(z)^{1/2}
	  = \frac{1}{2\pi } \left[ (z-z_1)(z-z_2) \right]^{1/2} (z-z_0) (z-z_3)(z-z_4), \]
with the $z_j$ as in \eqref{Q_zeros}. Using  $e = \frac{1}{2}(a-c)$ as in \eqref{eq_e}, we then get
\begin{align*} D'(x+ic) & = \frac{1}{2\pi} \sqrt{x^2-b^2} (x+i(c+a))((x+i(c-e))^2-d^2) 
\end{align*}
which implies that for $x < -b$,
\begin{align} \label{ReDprime} 
	\Re D'(x+ic) & = \frac{1}{2\pi} \sqrt{x^2-b^2} 
		\left(x^3  - ( \tfrac{21}{4} c^2 + \tfrac{1}{2} ac - \tfrac{3}{4} a^2 + d^2)x \right) \\
			\label{ImDprime}
	\Im D'(x+ic) & = \frac{1}{2 \pi} \sqrt{x^2-b^2} 
		\left(4 cx^2 + \tfrac{1}{4}(-a-c)(a^2-6ac+9c^2 +4 d^2) \right).
\end{align}

Now we take the values $a=a_1$, $b=b_1$, etc., as given in Lemma \ref{LQ2}.
Then it turns out that the coefficients in the quadratic expression in \eqref{ImDprime} 
are both $> 0$ and therefore
\begin{equation} \label{ImDprime2} 
	\Im D'(x+ic_1)  > 0, \qquad x < -b_1. 
	\end{equation}
The cubic expression in \eqref{ReDprime} is approximately $x^3  -  0.5171  x$
and this is negative for $x < -b_1 \approx -1.34$. Thus
\begin{equation} \label{ReDprime2} 
	\Re D'(x+ic_1)  < 0, \qquad x < -b_1. 
	\end{equation}
Since $D(z_1) = D(-b_1 + ic_1) = 0$, it follows from the above that
$\Im D(z) < 0$ and $\Re D(z) > 0$ for $z= x+ ic_1$, $x < -b_1$.
Thus $\alpha_1$, $\beta_1$, and $\Gamma_{1,1}$ do not intersect $L$, as claimed.

From this we conclude that $\alpha_1$ and $\Gamma_{1,1}$ are  in the domain  
lying above $L \cup \gamma_1$
in the left half-plane, while $\beta_1$ is in the domain below  $L \cup \gamma_1$
in the left half-plane, see also Figure \ref{crit_ext2}. Thus these curves tend to infinity in these domains, or
end at another zero of $Q$.

The double zero $z_3 = -d_1 + e_1 i$ is the only other zero of $Q$ that is in the left half-plane.
It lies in the domain below $L \cup \gamma_1$,
as is clear from the values of $d_1$ and $e_1$ in  Lemma \ref{LQ2}. Therefore $\alpha_1$ and $\Gamma_{1,1}$
tend to infinity in the domain above $L \cup \gamma_1$. They enclose an unbounded domain
that is free from zeros of $Q$. At $z_1$ there is an angle $\frac{\pi}{3}$. 
By Teichm\"uller's lemma, see Theorem \ref{Teichmuller}, the domain makes an angle $\frac{\pi}{10}$ at infinity. 
This leaves us with two possibilities.
Either $\alpha$ ends at an angle $\theta_3$ and $\Gamma_{1,1}$ ends at angle $\theta_3 + \frac{\pi}{10} = \epsilon_3$
at infinity, or $\alpha$ ends at an angle $\theta_4$ and $\Gamma_{1,1}$ ends at angle 
$\theta_4 + \frac{\pi}{10} = \epsilon_4$ at infinity.

Now for $\beta_1$ there are three possibilities depending on how it is situated with respect to the double zero $z_3$ 
of $Q$.\footnote{It is Case I that actually happens, as a numerical computation of $\beta_1$ shows.}
\begin{description}
\item[Case 1:] $\beta_1$ ends at infinity and the domain bounded by $\alpha_1$ and $\beta_1$
does not contain $z_3$.
\item[Case 2:] $\beta_1$ ends at infinity and the domain bounded by $\alpha_1$ and $\beta_1$
contains $z_3$.
\item[Case 3:] $\beta_1$ ends at the double zero $z_3$ of $Q$.
\end{description}

In Cases 1 and 2, the trajectory $\beta_1$ tends to infinity at  one of the angles 
$\theta_5$, $\theta_{6}$ or $\theta_7$, see Figure \ref{dir_infty},
since it remains below the half-line $L$ in the left half-plane. Then we have a domain $\Omega$ that is 
bounded by $\alpha_1$ and $\beta_1$, and that makes an angle $\frac{2\pi}{3}$ at $z_1$. 
In Case 1 there is no zero of $Q$ in $\Omega$ and it follows from 
 Theorem \ref{Teichmuller} that $\alpha_1$ and $\beta_1$ end at infinity in consecutive angles $\theta_j$
and $\theta_{j +1}$. Then the only possibility is that $\alpha_1$ ends at angle $\theta_4$ and
$\beta$ ends at angle $\theta_5$. In that case $\Gamma_{1,1}$ ends at angle $\epsilon_4$.

In Case 2 the double zero $z_3$ is in $\Omega$. Then by Theorem \ref{Teichmuller}
the trajectories $\alpha_1$ and $\beta_1$ end at angles $\theta_j$ and $\theta_j + \frac{3\pi}{5} = \theta_{j+3}$
for some $j$. 
Thus if $\alpha_1$ ends at angle $\theta_3$ then $\beta_1$ ends at angle $\theta_6$
and if $\alpha_1$ ends at angle $\theta_4$ then $\beta_1$ ends at angle $\theta_7$.
Since $z_3$ is a double zero there are four trajectories emanating from $z_3$. These trajectories
cannot intersect with $\alpha_1$ or $\beta_1$, and they cannot form closed loops either. So they have to
extend to infinity in four different directions, and these directions are bounded by the directions
at the angles $\theta_j$ and $\theta_{j+3}$ of $\alpha_1$ and $\beta_1$. 
Thus from $z_3$ there is a trajectory ending at infinity at  each of the angles 
$\theta_j$, $\theta_{j+1}$, $\theta_{j+2}$ and $\theta_{j+3}$.
Since $\Im D \equiv \text{const}$ on these four trajectories from $z_3$ (with the same constant
on each of the trajectories), and since $\Im D$ is strictly increasing on the half-line $L$,
see \eqref{ImDprime2}, only one of the trajectories can intersect with $L$. It implies
that the angles $\theta_3$ and $\theta_4$ cannot both be reached by trajectories from $z_3$,
see Figure \ref{dir_infty}, which means that $\alpha_1$ cannot end at angle $\theta_3$.
Thus $\alpha_1$ ends at angle $\theta_4$ and then $\Gamma_{1,1}$ ends at angle $\epsilon_4$ also in Case 2.

In Case 3 the trajectory $\beta_1$ ends at $z_3$. Then we can continue $\beta_1$
with another  critical trajectory $\delta_1$ from $z_3$ that is necessarily unbounded. 
We do it in such a way
that we obtain a $Q$-polygon $\Omega$ bounded by $\alpha_1$, $\beta_1$ and $\delta_1$
that makes an inner angle $\frac{2\pi}{3}$ at $z_1$ and an inner angle $\frac{\pi}{2}$
at $z_3$. Then $\Im D = 0$ on $\delta_1$ and so $\delta_1$ does not intersect the half-line $L$.
It then ends at infinity at an angle $\theta_5$, $\theta_{6}$ or $\theta_{7}$.
Since $\Omega$ does not contain any zeros, we  find by Theorem~\ref{Teichmuller}
that $\alpha_1$ and $\delta_1$ make an angle $\frac{\pi}{5}$ at infinity. Then
just as in Case 1 we conclude that $\alpha_1$ ends at angle $\theta_4$ and
therefore $\Gamma_{1,1}$ ends at angle $\epsilon_4$. 

Thus in all cases we have that $\Gamma_{1,1}$ ends at infinity at an angle $\epsilon_4$.
Because of the  symmetry in the imaginary axis, $\Gamma_{2,1}$ then ends at $\epsilon_0$, 
and it follows that $\Gamma_1$ belongs to  the class $\mathcal T_{3,1}$.
\end{proof}

An entirely similar method of proof shows the corresponding result for case $p=2$.
\begin{Lem}\label{extension_V2_b}
The critical trajectory $\gamma_2$ from $z_1$ to $z_2$ has an extension to a curve
\[
 \Gamma_2 = \Gamma_{2,1} \cup \gamma_2 \cup \Gamma_{2,2}
\]
where for $i=1,2$, $\Gamma_{2,i}$ is a vertical trajectory of $-Q_2(z) dz^2$ that 
extends from $z_i$ to infinity.

The curve $\Gamma_2$ belongs to the class $\mathcal T_{4,5}$.
\end{Lem}

\subsection{Proof of Theorem \ref{Thm_single_cut_V2}}

Having Lemmas \ref{LQ2}, \ref{extension_V2}, and \ref{extension_V2_b} 
we can now complete the proof of Theorem \ref{Thm_single_cut_V2},
in a similar way as we completed the proofs of parts (a) and (b) of Theorem \ref{Thm_single_cut_V1}.

\begin{proof}
For both cases $p=1$ and $p=2$, Lemma \ref{LQ2} establishes that there is a critical trajectory
 $\gamma_p$ that connects the zeros of $Q_p(z)$.

By Lemma \ref{extension_V2} the trajectory $\gamma_1$ has analytic extension to an unbounded contour
\[ \Gamma_1 = \Gamma_{1,1} \cup \gamma_1 \cup \Gamma_{1,2} \]
in the class $\mathcal T_{3,1}$, and by Lemma \ref{extension_V2_b}
the trajectory $\gamma_2$ has analytic extension to an unbounded contour
\[ \Gamma_2 = \Gamma_{2,1} \cup \gamma_2 \cup \Gamma_{2,2} \]
in the class $\mathcal T_{4,5}$. 

Then as in the proof of Theorem \ref{Thm_single_cut_V1} we can apply
the method in the proof of \cite[Theorem 2.2]{DHK} to show that 
$\Gamma_p$ has the $S$-property in the external field $\Re V$
and that $\supp(\mu_{\Gamma_p}) = \gamma_p$ for $p=1,2$. Thus by Theorem \ref{theoremGR} 
the zeros of $P_n$ accumulate on $\gamma_p$ as $n \to \infty$.
\end{proof}

\subsection*{Acknowledgements}
The authors are supported by FWO Flanders projects G.0641.11.

The second author is also supported by KU Leuven Research Grant OT/12/073, 
the Belgian Interuniversity Attraction Pole P07/18, FWO Flanders projects
G.0641.11 and G.0934.13, and by Grant No. MTM2011-28952-C02 of the Spanish 
Ministry of Science and Innovation.

\end{document}